\documentclass[a4paper]{article}
\usepackage[margin=3cm]{geometry}
\usepackage{amsmath,amssymb,amsthm}
\usepackage[colorlinks=true,linkcolor=blue,citecolor=blue]{hyperref}
\usepackage{graphicx} 
\usepackage[subpreambles=true]{standalone} 
\usepackage{import} 
\usepackage{tikz}
\usepackage{tikz-cd}
\usepackage{pgfplots}
\pgfplotsset{compat=1.18}

\title{Defining extended TQFTs via handle attachments}
\author{Benjamin Ha\"ioun }
\date{\today}

\newtheorem{counter}{Counter}
\newtheorem{theorem}[counter]{Theorem}
\newtheorem{lemma}[counter]{Lemma}
\newtheorem{corollary}[counter]{Corollary}
\newtheorem{proposition}[counter]{Proposition}
\theoremstyle{definition}
\newtheorem{definition}[counter]{Definition}
\newtheorem{remark}[counter]{Remark}
\newtheorem{example}[counter]{Example}

\numberwithin{counter}{section}
\numberwithin{equation}{section}

\newcommand\D{\mathbb{D}}
\renewcommand\S{\mathbb{S}}
\newcommand\Rr{\mathbb{R}}

\newcommand\CC{\mathcal{C}}

\newcommand\ZZ{\mathcal{Z}}
\newcommand\inj{\hookrightarrow}
\newcommand{\Cob}{\operatorname{\textsc{\textbf{Cob}}}}
\DeclareMathOperator{\cob}{Cob}

\DeclareMathOperator{\Hom}{Hom}
\DeclareMathOperator{\id}{id}
\DeclareMathOperator{\im}{Im}

\begin{document}

\maketitle
\begin{abstract}
We give a finite presentation of the symmetric monoidal bicategory of (smooth, oriented) closed manifolds, cobordisms and cobordisms with corners as an extension of the bicategory of closed manifolds, cobordisms and diffeomorphisms. The generators are the standard handle attachments, and the relations are handle cancellations and invariance under reversing the orientation of the attaching spheres.

In other words, given a categorified TQFT and 2-morphisms associated to the standard handles satisfying our relations, we construct a once extended TQFT.
\end{abstract}
\tableofcontents

\section{Introduction}
Topological Quantum Field Theories (TQFTs) were introduced by Atiyah and Segal as theories that assign vector spaces (called state spaces in physics) to closed manifolds of a certain ``space" dimension $n$, and linear maps between these vector spaces (called correlation functions in physics) to cobordisms of dimension $n+1$, the ``space-time" dimension. In modern language, we say that an $(n+1)$-TQFT is a symmetric monoidal functor 
$$\ZZ:\cob_{n+1}\to \operatorname{Vect}$$
from the 1-category of cobordisms $\cob_{n+1}$ to vector spaces.

A common strategy in defining such TQFTs is to find a convenient way of presenting $(n+1)$-manifolds, constructing correlation functions from this presentation, and checking that this construction does not depend on the chosen presentation.

For example, the Turaev--Viro \cite{TuraevViroStateSum} and Crane--Yetter \cite{CraneYetter} constructions are based on triangulations of 3 and 4 manifolds. The Witten--Reshetikhin--Turaev \cite{ReshetikhinTuraev, TuraevBook, WittenJonesPol} theories are based on surgery presentations of 3-manifolds. It is argued in \cite{WalkerOnWitten, WalkerNotes} that a handle-attachment approach towards each of these theories is possible and insightful. 

Checking invariance of the construction under a change of presentation is a priori not an easy task. There are some powerful results simplifying it significantly for closed manifolds: any two triangulations are related by Pachner moves, any two surgery presentations are related by Kirby moves, any two handle decompositions are related by handle cancellations and handle slides. However, when one has a fixed parametrization of the boundary, as morphisms in the cobordism category do, one has to be a little bit more careful. The reader may want to take a look at the various kinds of critical points for Morse functions on a manifold with boundary considered in \cite[Section 7]{WalkerOnWitten}.

A very handy result dealing with these subtleties has appeared in \cite{Juhasz} where the author gives a presentation of the cobordism category. Note that this presentation is infinite as is contains all the ways one can attach a handle on an $n$-manifold. Using \cite[Thm. 1.8]{Juhasz}, if one already knows the state spaces of the theory and how they behave under diffeomorphisms, then defining a TQFT amounts to giving linear maps associated to each handle attachment, and checking four explicit relations. 
Non-semisimple version of Turaev--Viro theories \cite{CGPVnc2plus1} and of Crane--Yetter theories \cite{CGHP} have been recently defined using Juh\`asz's results.

Extended theories are theories where one can also cut and glue in the spaces directions, and not only in the time direction. They capture the notion of locality in physics. Once-extended theories are ones where we can cut in one space direction. To simplify notation we will denote the space-time dimension by $n+2$. In some sense, there are now two time directions $s$ and $t$, an $n$-manifold can evolve in the direction of time $s$ through an $(n+1)$-space, and an $(n+1)$-space can evolve in the direction of time $t$ through an $(n+2)$-dimensional space-time. Note that one has to consider $(n+2)$-cobordisms with corners in this picture. In modern language, one say that a once-extended TQFT, or an $(n+1+1)$-TQFT, is a symmetric monoidal 2-functor 
$$\ZZ: \Cob_{n+1+1}\to 2\operatorname{Vect}$$

\subsection{Results}
The main objective of this paper is to give a result analogous to \cite[Thm 1.8]{Juhasz} for once-extended theories. We suppose we are given a TQFT defined only on $n$-manifolds, $(n+1)$-cobordisms and their diffeomorphisms (which assemble into a bicategory $\Cob_{n+1+\varepsilon}$). This is called a categorified $(n+1)$-TQFT, or and $(n+1+\varepsilon)$-TQFT, in the literature, and we want to extend it to $(n+2)$-cobordisms. We show that such an extension is equivalent to a prescription on what to do on handle attachments satisfying some explicit relations. This is formulated precisely in:

\newtheorem*{Thm_ETQFTfromHandle}{Theorem \ref{Thm_ETQFTfromHandle}}
\begin{Thm_ETQFTfromHandle}\it
    Given a categorified TQFT $\ZZ^\varepsilon: \Cob_{n+1+\varepsilon}\to \CC$ with values in some symmetric monoidal bicategory $\CC$, $n\geq 0$, and 2-morphisms $$\left(Z_k: \ZZ^\varepsilon(S^{k-1}\times \D^{n+2-k})\to \ZZ^\varepsilon(\D^k\times S^{n+1-k})\right)_{k\in \{0,\dots,n+2\}}$$ such that \begin{itemize}
        \item[(a)] each pair $Z_k, Z_{k+1},\ k \in \{0,\dots,n+1\}$, satisfy the handle cancellation in the sense of Definition \ref{Def_SatisfyHdlCancel}, and
        \item[(b)] each $Z_k,\ k \in \{1,\dots,n+1\}$, is $\iota$-invariant in the sense of Definition \ref{Def_iotaInvariant}
    \end{itemize}
    there exists a unique $(n+1+1)$-TQFT \begin{equation*}
        \ZZ: \Cob_{n+1+1}\to \CC
    \end{equation*}
    extending $\ZZ^\varepsilon$ such that $\ZZ(H_k)=Z_k$, where $H_k$ is the handle of index $k$ of Definition \ref{Def_handle}.

    Reciprocally, given an $(n+1+1)$-TQFT $\ZZ$, taking $\ZZ^\varepsilon$ to be the restriction of $\ZZ$ to $\Cob_{n+1+\varepsilon}$, the morphisms $Z_k:= \ZZ(H_k)$ satisfy $(a)$ and $(b)$.

    In other words, $\Cob_{n+1+1}$ is the symmetric monoidal bicategory obtained from $\Cob_{n+1+\varepsilon}$ by adjoining 2-morphisms $(H_k)_{k\in \{0,\dots,n+2\}}$ modulo the relations (a) and (b).
\end{Thm_ETQFTfromHandle}

The idea of defining extended TQFTs using handle attachments as above is certainly not new. It appears in Lurie's sketch of proof of the cobordism hypothesis \cite{LurieCob} in the construction of fully extended TQFTs, and in Walker's construction of skein-theoretic TQFTs in \cite{WalkerNotes}. Careful and complete proofs were however missing. 
Note also closely related work-in-progress of David Reutter and Kevin Walker announced in \cite{ReutterSlides, WalkerSlides} where the two authors propose an approach to defining TQFTs one handle at a time, by verifying a certain non-degeneracy condition on each handle attachment, which is reminiscent of our condition (a).

Note that, unlike the presentation of \cite{Juhasz} in the 1-categorical case, our presentation of the cobordism bicategory as a symmetric monoidal extension of $\Cob_{n+1+\varepsilon}$ is finite. This is not too surprising as the $k$-handle of Definition \ref{Def_handle} can only be seen as a cobordism with corners.

Note however that we do not try to give generators and relations for the 1-morphisms of $\Cob_{n+1+1}$, but only for 2-morphisms. A full presentation of $\Cob_{n+1+1}$ (not as an extension of $\Cob_{n+1+\varepsilon}$) is a much greater task. Indeed, in our setting we only need to know when two handle decompositions give diffeomorphic cobordisms, and we do not need to remember this diffeomorphism. If one attempts to decompose 1-morphisms, then indeed one needs to remember this data. In Morse theory terms, one not only needs to consider paths of Morse functions, i.e. Cerf theory, but also paths of paths, and classify all types of singularities that may appear there. This has been done in dimension 2 in \cite{SPPhD} for $\Cob_{0+1+1}$, and will appear in any dimension for $\Cob_{n+1+\varepsilon}$ in \cite{Filippos}. Combining the present paper with \cite{Filippos}, one would expect to obtain a full presentation of $\Cob_{n+1+1}$.

An important application of the general methods we develop herein is to construct a once-extended 4-3-2 skein theory for non-semisimple ribbon tensor categories, i.e. a once-extended version of \cite{CGHP}. This application is the focus of a forthcoming paper.
The constructions of \cite{CGPVnc2plus1, CGHP} sometimes produce partially-defined TQFTs, called non-compact TQFTs, and we also give a non-compact version of Theorem \ref{Thm_ETQFTfromHandle}, simply dropping the last handle $Z_{n+2}$ and the relation involving it.

Finally, we show that an extension $\ZZ$ of a categorified TQFT $\ZZ^\varepsilon$ is completely determined by its value on the 0-handle. This is not too surprising. It already features in examples \cite{WalkerNotes, CGPVnc2plus1, CGHP} and is at the heart of the Reutter--Walker approach mentioned above. It is also expected from the proof of the cobordism hypothesis \cite[Prop. 3.4.19]{LurieCob}.
\newtheorem*{Thm_Classif}{Theorem \ref{Thm_Classif}}
\begin{Thm_Classif}\it
    Let $\ZZ,\ZZ': \Cob_{n+1+1}\to \CC$ (resp. $\Cob_{n+1+1}^{nc}\to \CC$) be (resp. non-compact) once-extended TQFTs which agree on $\Cob_{n+1+\varepsilon}$ and such that $\ZZ(H_0)=\ZZ'(H_0)$. Then $\ZZ=\ZZ'$.
\end{Thm_Classif}
One may interpret this result as saying that the existence of a once-extended TQFT $\ZZ$ extending $\ZZ^\varepsilon$ and with $\ZZ(H_0) = Z_0$ is a property of $Z_0$, and no extra structure. 

This result also helps shed some light on the $\iota$-invariance requirement of Theorem \ref{Thm_ETQFTfromHandle}. We show that the requirements (a) and (b) of Theorem \ref{Thm_ETQFTfromHandle} are equivalent to requiring the cancellation (a) for canceling sphere that intersect both positively and negatively.

\subsection{Outline}
In Section \ref{Sec_CobBicat} we define the different cobordisms bicategories we will need and set the terminology for once-extended, categorified and non-compact TQFTs.

In Section \ref{Sec_MainResult} we define the generators and the relations associated to the cobordism bicategory and state the main result.

In Section \ref{Sec_MorseDatum} we prove a result analogous to \cite[Thm. 1.7]{Juhasz} for cobordisms with corners, Theorem \ref{Thm_JuhaszCorner}, which we will use to construct our once-extended TQFT on Hom categories. The proof is based on Morse and Cerf theory, and follows very standard methods \cite{Milnor_hcobordism, Cerf, GWW13, Juhasz}, adapted to corners. This adaptation is straightforward after introducing the appropriate notion of Morse datum for cobordisms with corners. The author does not know a reference where it is studied though, so a complete account is given.

In Section \ref{Sec_2Fun} we show that the data in Theorem \ref{Thm_ETQFTfromHandle} induces the data needed in Theorem \ref{Thm_JuhaszCorner} and we assemble all the functors thus constructed on Hom categories into a symmetric monoidal 2-functor. This section exhibits the really new behavior of our setting, reducing Juh\'asz infinite presentation to a finite one.

In Section \ref{Sec_Classif} we prove that an extension is determined by its value on the 0-handle.

\subsection*{Acknowledgements} I am grateful to David Jordan and Filippos Sytilidis for enlightening conversations. I also thank the anonymous referee for useful comments, and in particular for making me add many pictures. This research took place at the University of Edinburgh and is funded by the Simons Foundation award 888988 as part of the Simons Collaboration on Global Categorical Symmetry.

\section{The cobordism bicategory}\label{Sec_CobBicat}
In this section we recall the definition of the bicategory of cobordisms equipped with collars. These collars, though there is a contractible space of choices, are essential to explicitly define horizontal composition. We follow \cite{SPPhD, Scheimbauer, StolzTeichner} with some simplifications in our context.

To help the reader appreciate the difference between the 1-category $\cob_{n+2}$ of $(n+1)$-manifolds and $(n+2)$-cobordisms between them, and the bicategory $\Cob_{n+1+1}$ of $n$-manifolds, $(n+1)$-cobordisms and $(n+2)$-cobordisms with corner, we will always use bold $\Cob$ for cobordism bicategories, and plain $\cob$ for cobordism 1-categories.

\begin{definition}\label{Def_CobCorner}
    Let $\Sigma_-,\ \Sigma_+$ be closed oriented smooth $n$-manifolds. The category $\cob_{n+2}^{\Sigma_-,\Sigma_+}$ has 
    \begin{description}
        \item[objects:] oriented smooth $(n+1)$-manifolds $M$ with boundary $\partial M \simeq \overline{\Sigma_-}\sqcup \Sigma_+$ equipped with a \textbf{side collar} $\Sigma_- \times [-1,-\frac{1}{2}) \sqcup \Sigma_+ \times (\frac{1}{2},1] \inj M$ of its boundary. By abuse, we will usually abbreviate this as $\Sigma_\pm \times [\pm1,\pm\frac{1}{2}) \inj M$. We require that these collars are regular in the sense that they can be extended to collars $\Sigma_- \times [-1,0) \sqcup \Sigma_+ \times (0,1] \inj M$. We do not choose such an extension, but note that $\Sigma_\pm \times [\pm1,\pm\frac{1}{2}] \inj M$ is uniquely defined by continuity.
        \item[morphisms:] Cobordisms with collared corners, i.e. oriented smooth $(n+2)$-manifolds with corners $W$ equipped with
        \begin{itemize}
            \item a decomposition of their boundary $\partial W = \partial_0W \cup \partial_1W$ such that $\partial_0W \cap \partial_1W$ is the set of corner points of $W$\footnote{This is also called a structure of $\langle 2\rangle$-manifold.},
            \item an identification of the \textbf{side boundary} $\partial_0W\simeq (\overline{\Sigma_-}\sqcup \Sigma_+)\times [-1,1]$ together with a regular collar $\Sigma_\pm \times [\pm1,\pm\frac{1}{2}) \times [-1,1] \inj W$,
            \item a \textbf{source and target diffeomorphism} $\partial_1W \simeq \overline{M}\sqcup M'$ respecting the side collars. 
        \end{itemize} 
        See Figure \ref{fig:cobwithcorners}.\\
        These are considered up to diffeomorphisms $\Phi:W\to W'$ preserving the source and target diffeomorphisms and preserving the side collars up to reparametrization of the $[-1,1]$-coordinate. More precisely, we ask that there exists a reparametrization $\varphi:[-1,1]\tilde\to[-1,1]$ which is the identity near $\{-1,1\}$ such that a point $(x,t,s) \in \Sigma_\pm \times [\pm1,\pm\frac{1}{2}) \times [-1,1] \inj W$ is mapped by $\Phi$ to $(x,\varphi(t),s)\in \Sigma_\pm \times [\pm1,\pm\frac{1}{2}) \times [-1,1] \inj W'$. 
        
        \item[composition $M_1\overset{W}{\to}M_2\overset{W'}{\to}M_3$:] is given by gluing along $M_2$, i.e. $W'\circ W := W \cup_{M_2} W'$. It is equipped with a smooth structure by choosing collars of $M_2$ in both $W$ and $W'$ compatible with the side collars. This smooth structure is well-defined up to diffeomorphism of $W \cup_{M_2} W'$ preserving the side collars and the source and target diffeomorphisms \cite[Thm. 1.4]{Milnor_hcobordism} or \cite[Lemma 6.1]{MunkresDiffTop}. The side collars of the composition are given by gluing the side collars of $W$ and $W'$ and reparametrizing $[-1,1]\underset{1=-1}{\cup}[-1,1]\simeq [-2,2] \Tilde{\to} [-1,1]$ by a diffeomorphism $r$ which is a translation near the boundary. This is readily checked to be associative up to the notion of diffeomorphism introduced above.\footnote{The composition $[-1,1]\underset{1=-1}{\cup}[-1,1]\Tilde{\to} [-1,1]$ will never be associative on the nose, and this is why we needed to allow diffeomorphisms of $W$ that preserve the side collars only up to reparametrization.}
    \end{description}
    \begin{figure}
        \centering
\begin{tikzpicture}[xscale = 1.5]
\draw[very thick] (-0.9,0.5) -- (-0.4,0.5);
\draw[very thick] (0.6,0.5) -- (1.1,0.5);
\draw (-0.5,0) arc(-90:0:0.3 and 0.25) arc(0:90:0.2 and 0.25);
\draw (0.5,0) arc(-90:-180:0.2 and 0.25) arc(180:90:0.3 and 0.25);
\draw (-1,2) --++(2,0);
\draw (-0.9,2.5) --++(2,0);
\draw (-0.9,0.5)  -- ++(0,2);
\fill[gray!30, opacity = 0.5] (-1,0) rectangle ++(0.5,2);
\fill[gray!30, opacity = 0.5] (-0.9,0.5) rectangle ++(0.5,2);
\fill[gray!30, opacity = 0.5] (0.5,0) rectangle ++(0.5,2);
\fill[gray!30, opacity = 0.5] (0.6,0.5) rectangle ++(0.5,2);
\draw (-1,0) -- ++(0,2);
\draw (1,0) -- ++(0,2);
\draw (1.1,0.5) -- ++(0,2);
\draw[very thick] (-1,0) -- ++(0.5,0);
\draw[very thick] (0.5,0) -- (1,0);
\draw[very thick] (-1,2) -- (-0.5,2);
\draw[very thick] (-0.9,2.5) -- (-0.4,2.5);
\draw[very thick] (0.5,2) -- (1,2);
\draw[very thick] (0.6,2.5) -- (1.1,2.5);
\draw (-0.2,0.25) arc(180:0:0.25 and 1);
\node at (0,1.5){$W$};
\begin{scope}[yshift = 1.5cm]
\draw[very thick] (-1,2) -- (-0.5,2);
\draw[very thick] (-0.9,2.5) -- (-0.4,2.5);
\draw[very thick] (0.5,2) -- (1,2);
\draw[very thick] (0.6,2.5) -- (1.1,2.5);
\draw (-1,2) --++(2,0);
\draw (-0.9,2.5) --++(2,0);
\node at (0,2.9){$M_+$};
\end{scope}
\begin{scope}[yshift = -1.5cm]
\draw[very thick] (-1,0) -- (-0.5,0);
\draw[very thick] (-0.9,0.5) -- (-0.4,0.5);
\draw[very thick] (0.5,0) -- (1,0);
\draw[very thick] (0.6,0.5) -- (1.1,0.5);
\draw (-0.5,0) arc(-90:0:0.3 and 0.25) arc(0:90:0.2 and 0.25);
\draw (0.5,0) arc(-90:-180:0.2 and 0.25) arc(180:90:0.3 and 0.25);
\node at (0,-0.2){$M_-$};
\end{scope}
\begin{scope}[xshift = 1.5cm]
\fill[gray!30, opacity = 0.5] (0.5,0) rectangle ++(0.5,2);
\fill[gray!30, opacity = 0.5] (0.6,0.5) rectangle ++(0.5,2);
\draw (1,0) -- ++(0,2);
\draw (1.1,0.5) -- ++(0,2);
\draw[very thick] (0.5,0) -- (1,0);
\draw[very thick] (0.6,0.5) -- (1.1,0.5);
\draw[very thick] (0.5,2) -- (1,2);
\draw[very thick] (0.6,2.5) -- (1.1,2.5);
\node[anchor = west, xscale = 0.8] at (1.2,1){\small $\Sigma_+ \times (\frac{1}{2},1] \times [-1,1]$};
\end{scope}
\begin{scope}[xshift = -1.5cm]
\fill[gray!30, opacity = 0.5] (-1,0) rectangle ++(0.5,2);
\fill[gray!30, opacity = 0.5] (-0.9,0.5) rectangle ++(0.5,2);
\draw (-1,0) -- ++(0,2);
\draw (-0.9,0.5)  -- ++(0,2);
\draw[very thick] (-1,0) -- (-0.5,0);
\draw[very thick] (-0.9,0.5) -- (-0.4,0.5);
\draw[very thick] (-1,2) -- (-0.5,2);
\draw[very thick] (-0.9,2.5) -- (-0.4,2.5);
\node[anchor = east, xscale = 0.8] at (-1.2,1){\small $\Sigma_- \times [-1,\frac{1}{2}) \times [-1,1]$};
\end{scope}
\begin{scope}[xshift = -1.5cm, yshift = -1.5cm]
\draw[very thick] (-1,0) -- (-0.5,0);
\draw[very thick] (-0.9,0.5) -- (-0.4,0.5);
\node[anchor = east] at (-1.2,0.25){\small $\Sigma_- \times [-1,\frac{1}{2})$};
\end{scope}
\begin{scope}[xshift = 1.5cm, yshift = -1.5cm]
\draw[very thick] (0.5,0) -- (1,0);
\draw[very thick] (0.6,0.5) -- (1.1,0.5);
\node[anchor = west] at (1.2,0.25){\small $\Sigma_+ \times (\frac{1}{2},1]$};
\end{scope}
\begin{scope}[xshift = -1.5cm, yshift = 1.5cm]
\draw[very thick] (-1,2) -- (-0.5,2);
\draw[very thick] (-0.9,2.5) -- (-0.4,2.5);
\node[anchor = east] at (-1.2,2.25){\small $\Sigma_- \times [-1,\frac{1}{2})$};
\end{scope}
\begin{scope}[xshift = 1.5cm, yshift = 1.5cm]
\draw[very thick] (0.5,2) -- (1,2);
\draw[very thick] (0.6,2.5) -- (1.1,2.5);
\node[anchor = west] at (1.2,2.25){\small $\Sigma_+ \times (\frac{1}{2},1]$};
\end{scope}
\node at (-1.5,3.75) {$\inj$};
\node[xscale = -1] at (1.5,3.75) {$\inj$};
\node at (-1.5,-1.25) {$\inj$};
\node[xscale = -1] at (1.5,-1.25) {$\inj$};
\node at (-1.5,1.25) {$\inj$};
\node[xscale = -1] at (1.5,1.25) {$\inj$};
\node[rotate = -90] at (-2.15,3) {$\inj$};
\node[rotate = -90] at (2.15,3) {$\inj$};
\node[rotate = -90] at (0,3) {$\inj$};
\node[rotate = 90] at (-2.15,-0.5) {$\inj$};
\node[rotate = 90] at (2.15,-0.5) {$\inj$};
\node[rotate = 90] at (0,-0.5) {$\inj$};
\end{tikzpicture}        
\caption{A cobordism with collared corners.}
        \label{fig:cobwithcorners}
    \end{figure}
    An orientation-preserving diffeomorphism $f:M\to M'$ which preserves side collars induces a morphism $W_f:=M'\times [-1,1]$ with side collars induced by the ones in $M'$, source diffeomorphism $f$ and target diffeomorphism the identity. 

    The identity on $M$ is the morphism $W_{\id}$.
\end{definition}

\begin{definition}
    The symmetric monoidal cobordism bicategory $\Cob_{n+1+1}$ has:
    \begin{description}
        \item[Objects: ]  Closed oriented smooth $n$-manifolds $\Sigma$
        \item[Hom category $\Sigma_- \to \Sigma_+$: ] the category $\cob_{n+2}^{\Sigma_-,\Sigma_+}$ defined above.

        A diffeomorphism $f:\Sigma_-\to\Sigma_+$ induces a 1-morphism $M_f := \Sigma_+\times [-1,1]$ with side collar $(f\times\id)\sqcup \id: \Sigma_- \times [-1,-\frac{1}{2}) \sqcup \Sigma_+ \times (\frac{1}{2},1] \inj M_f$.

        \item[Identity 1-morphism on $\Sigma$:] $\id_\Sigma := M_{\id}$ induced by the identity diffeomorphism.
        
        \item[Composition of 1-morphisms $\Sigma_1 \overset{M}{\to}\Sigma_2\overset{M'}{\to}\Sigma_3$:] is given by gluing of the collars
        \begin{equation*}
            M'\circ M := M \underset{\Sigma_2\times I}{\cup}M'
        \end{equation*} where $I=[0,1]$ is canonically identified to $[-1,-\frac{1}{2}]$ or $[\frac{1}{2},1]$, with side collars for $\Sigma_1$ and $\Sigma_3$ inherited from those of $M$ and $M'$.

        \item[Horizontal composition of 2-morphisms $M_1 \overset{W}{\longrightarrow}M_2$ and $M_1'\overset{W'}{\longrightarrow}M_2'$:] where $M_1,M_2:\Sigma_1\to\Sigma_2$ and $M_1',M_2':\Sigma_2\to\Sigma_3$ is given by gluing of the collars 
        \begin{equation*}
            W'\circ_h W := W \underset{\Sigma_2\times I \times [-1,1]}{\cup}W'
        \end{equation*} 
        where $I$ is canonically identified to $[-1,-\frac{1}{2}]$ or $[\frac{1}{2},1]$, with side collars and source and target diffeomorphisms inherited from those of $W$ and $W'$.
        

        \item[Monoidal structure:] disjoint union on objects, 1- and 2-morphisms with projections and source and target diffeomorphisms induced by the universal property of coproducts. The unit is $\emptyset$.

        \item[Associators for $\Sigma_1 \overset{M_{12}}{\to}\Sigma_2\overset{M_{23}}{\to}\Sigma_3\overset{M_{34}}{\to}\Sigma_4$:] there is a diffeomorphism
        \begin{equation*}
            \left(M_{12} \underset{\Sigma_2\times I}{\cup}M_{23}\right)\underset{\Sigma_3\times I}{\cup}M_{34} \simeq M_{12} \underset{\Sigma_2\times I}{\cup}\left(M_{23}\underset{\Sigma_3\times I}{\cup}M_{34}\right)
        \end{equation*} induced by the universal property of colimits\footnote{The gluing operation defining composition is a colimit in topological spaces.} which induces a 2-morphism. 

        \item[Unitors for $\Sigma\overset{\id_\Sigma}{\to} \Sigma\overset{M}{\to}\Sigma'$:] Remember that we required that the side collars are regular in the sense that they can be extended to collars $\Sigma_- \times [-1,0) \sqcup \Sigma_+ \times (0,1] \overset{\varphi}{\inj} M$. This extension is unique up to isotopy.
        
        The composition $M\circ \id_\Sigma = \Sigma\times [-1,1]\underset{\Sigma\times I}{\cup} M$ contains
        \begin{equation*}
        \psi:         \Sigma \times [-\frac{5}{2},0) \simeq \Sigma \times \left( [-1,1] \underset{[\frac{1}{2},1]=[-1,-\frac{1}{2}]}{\cup}[-1,0)\right) \overset{\id \cup \varphi}{\inj} M\circ \id_\Sigma
        \end{equation*}
        whose complement is $M\smallsetminus \im(\varphi)$.
        
        Choose a (unique-up-to-isotopy) diffeomorphism $\rho: [-\frac{5}{2},0) \simeq [-1,0)$ which is a translation on a neighborhood on $[-\frac{5}{2}, -2+\delta]\sqcup [-\delta,0)\subseteq [-\frac{5}{2},0)$ for some $\delta>0$.
        
        We can now define a diffeomorphism:
        \begin{equation*}\begin{array}{rcl}
           r:  M\circ \id_\Sigma&\to& M\\
            x&\mapsto& \left\{ \begin{array}{ll}
            x & \text{if   } x\in M\smallsetminus \im(\varphi)\\
            \varphi(p,\rho(t)) & \text{if   } x=\psi(p,t) \in \im(\psi)
            \end{array}   \right.
        \end{array}\end{equation*}
 preserving side collars.
 
        The right unitor for $M$ is the 2-morphism induced by this diffeomorphism. It does not depend on the choice of $\varphi$ and $\rho$.
        Left unitors are defined similarly.

        More generally, given a diffeomorphism $\Sigma_-\overset{f}{\to}\Sigma_-'$ and a 1-morphism $M:\Sigma_-\to\Sigma_+$, this construction gives a diffeomorphism $M\circ M_f \simeq {}^fM$ preserving source and target diffeomorphisms, where ${}^fM$ is $M$ as a manifold but with side collar twisted by $f$.

        \item[Interchanger for the monoidal structure:] Given $\Sigma_1 \overset{M_{12}}{\to}\Sigma_2\overset{M_{23}}{\to}\Sigma_3$ and $\Sigma_1' \overset{M_{12}'}{\to}\Sigma_2'\overset{M_{23}'}{\to}\Sigma_3'$, we have a canonical diffeomorphism
        \begin{equation*}
           \phi: (M_{23}\circ M_{12}) \sqcup (M_{23}'\circ M_{12}') \simeq (M_{23} \sqcup M_{23}') \circ (M_{12} \circ M_{12}')
        \end{equation*} 
        given by the universal property of colimits, which induces a 2-morphism.
        
        \item[Associators and unitors for the monoidal structure:] are induced by the universal property of colimits. 

        \item[Symmetric braiding:] is induced by the flip of components $\Sigma\sqcup \Sigma'\simeq\Sigma'\sqcup\Sigma$ and $M\sqcup M'\simeq M'\sqcup M$. This is also the map induced by the universal property of coproducts. All the coherence modifications are between compositions of cobordisms induced by diffeomorphims where the composition of the diffeomorphisms agree on the nose. They are therefore induced by diffeomorphisms similar to the unitors above.
    \end{description}
\end{definition}
\begin{definition}
    The symmetric monoidal bicategory $\Cob_{n+1+1}^{nc}$ of non-compact cobordisms is the sub-bicategory of $\Cob_{n+1+1}$ with the same objects and 1-morphisms but only those 2-morphisms where the target diffeomorphism is surjective on connected components, i.e. every connected component of the $(n+2)$-cobordisms have non-empty outgoing boundary.
\end{definition}
\begin{definition}
    The symmetric monoidal bicategory $\Cob_{n+1+\varepsilon}$ has the same objects as $\Cob_{n+1+1}$, but 2-morphisms are isotopy classes of diffeomorphisms preserving the side collars. All the coherence data is the same as above, simply not taking cobordisms induced by the diffeomorphisms. We will denote $\cob_{(n+1)+\varepsilon}^{\Sigma_-,\Sigma_+} = \Hom_{\Cob_{n+1+\varepsilon}}(\Sigma_-,\Sigma_+)$ the category of collared $(n+1)$-cobordisms with prescribed boundary $\Sigma_-,\Sigma_+$ and diffeomorphisms preserving collars.

    It comes with an evident strictly symmetric monoidal strict 2-functor $\Cob_{n+1+\varepsilon}\to \Cob_{n+1+1}$ which is the identity on objects and 1-morphisms and maps a diffeomorphism $d$ to the induced cobordism $W_d$. This operation is well-defined as an isotopy $d\sim d'$ as above induces a diffeomorphism $W_d \simeq W_{d'}$ which preserves side collars and source and target diffeomorphisms. Note that this 2-functor factors through $\Cob_{n+1+1}^{nc}$.
\end{definition}
\begin{remark}
    The homotopy 1-category $h_1(\Cob_{n+1+\varepsilon})$ is equivalent to the usual category of cobordisms $\cob_{n+1}$ by mapping a collared cobordism $(M, \varphi: \Sigma_\pm \times [\pm1,\pm\frac{1}{2}) \inj M)$ to the cobordism $M\smallsetminus \varphi([\pm1, \pm\frac{3}{4}))$. Indeed, gluing of collars $M'\circ M := M \underset{\Sigma\times I}{\cup}M'$ identifies $\Sigma \times [\frac{1}{2},1] \inj M$ to $\Sigma \times [-1,-\frac{1}{2}] \inj M'$ and in particular glues $\Sigma \times\{\frac{3}{4}\}$ in $M$ to $\Sigma \times\{-\frac{3}{4}\}$ in $M'$.
\end{remark}
\begin{remark}
    One can also define a (strict) bicategory $\Cob_{n+\varepsilon+\varepsilon}$ with same objects, 1-morphisms being diffeomorphisms and 2-morphisms isotopies of diffeomorphisms considered up to higher isotopy. There is still a 2-functor $\Cob_{n+\varepsilon+\varepsilon} \to \Cob_{n+1+\varepsilon}\to\Cob_{n+1+1}$ but it is not strict and preserves composition only up to an isomorphisms similar to the unitors above.
\end{remark}
\begin{definition}
    Let $\CC$ be a symmetric monoidal bicategory. \\
    A \textbf{once-extended $(n+2)$-TQFT}, or an \textbf{$(n+1+1)$-TQFT}, with values in $\CC$ is a symmetric monoidal functor 
    \begin{equation*}
        \ZZ: \Cob_{n+1+1}\to \CC\ .
    \end{equation*}
    A \textbf{non-compact once-extended $(n+2)$-TQFT} is a symmetric monoidal functor 
    \begin{equation*}
        \ZZ: \Cob_{n+1+1}^{nc}\to \CC\ .
    \end{equation*}
    A \textbf{categorified $(n+1)$-TQFT}, or an \textbf{$(n+1+\varepsilon)$-TQFT}, is a symmetric monoidal functor 
    \begin{equation*}
        \ZZ: \Cob_{n+1+\varepsilon}\to \CC\ .
    \end{equation*}
    Given a possibly non-compact extended TQFT $\ZZ$ and a categorified TQFT $\ZZ^\varepsilon$, we say that \textbf{$\ZZ$ extends $\ZZ^\varepsilon$}, or that \textbf{$\ZZ^\varepsilon$ is the restriction of $\ZZ$ to $\Cob_{n+1+\varepsilon}$}, if $\ZZ^\varepsilon$ is equal to the composition $\Cob_{n+1+\varepsilon}\to \Cob_{n+1+1}^{nc} \overset{\ZZ}{\to} \CC$. 
\end{definition}
The main result of this paper is to show one can construct an $(n+1+1)$-TQFT from the data of an $(n+1+\varepsilon)$-TQFT together with coherent values on elementary handle attachments.

\section{Handle attachments and main result}\label{Sec_MainResult}
\subsection{Handle attachments}
We first introduce the standard $(n+2)$-dimensional $k$-handles. Note that the usual definition $\D^k\times \D^{n+2-k}: S^{k-1}\times \D^{n+2-k}\to \D^k\times S^{n+1-k}$ is not a valid 2-morphism in our definition of the cobordism bicategory. Indeed, it doesn't have a side boundary $S^{k-1}\times S^{n+1-k}\times [-1,1]$, but rather this side boundary has been pinched to a corner $S^{k-1}\times S^{n+1-k}$.

The following description of handles with non-pinched corners is very classical and is very natural from the point of view of Morse functions. Indeed, we are taking simply a standard Morse function $f$ and seeing the ball as a cobordism from the level set $f^{-1}(-1)$ to the level set $f^{-1}(1)$. However, this description is usually hidden in proofs, see e.g. \cite[Thm 3.12]{Milnor_hcobordism}, and not taken as the definition of the handle attachment itself, because it is much more involved to define than the pinched version.
\begin{definition}\label{Def_handle}
Let $k \in \{0,\dots,n+2\}$. Let $\D^k$ denote the standard closed $k$-dimensional disk of radius $1$, $S^{k-1}_r$ the $k$-dimensional sphere of radius $r$ and $S^{k-1}:=S^{k-1}_1$.

The $(n+1)$-manifolds $S^{k-1}\times \D^{n+2-k}$ and $\D^k\times S^{n+1-k}$ can both be seen as 1-morphism from $S^{k-1}\times S^{n+1-k}$ to the empty $n$-manifold. The side collar are given by identifying $S^{n+1-k}_r \subseteq \D^{n+2-k}$ with $S^{n+1-k}$ by rescaling for $-r \in [-1,-\frac{1}{2})$ (respectively identifying $S^{k-1}_r \subseteq \D^{k}$ with $S^{k-1}$ for $-r \in [-1,-\frac{1}{2})$).

The \textbf{$(n+2)$-dimensional index $k$ handle} is the 2-morphism $$H_k : S^{k-1}\times \D^{n+2-k}\to \D^k\times S^{n+1-k}$$ with
\begin{equation*}
    H_k := \left\{(x,y)\in\Rr^k\times\Rr^{n+2-k} \Big\vert\  -1\leq ||y||^2-||x||^2\leq 1 \text{  and  }  ||x||^2||y||^2 \leq 2\right\}
\end{equation*}
The source and target diffeomorphisms are respectively 
$$\begin{array}{rcl}
    S^{k-1}\times \D^{n+2-k} &\to& H_k\\
    (x,y) &\mapsto& (\sqrt{||y||^2+1}\, x,y) \text{  ,  and}
\end{array}$$
$$\begin{array}{rcl}
    \D^{k}\times S^{n+1-k} &\to& H_k\\
    (x,y) &\mapsto& (x,\sqrt{||x||^2+1}\, y)\ .
\end{array}$$
An explicit formula for the side collar is a bit tedious. Instead we notice that the manifold $H_k$ comes equipped with a Morse function $f: H_k \to \Rr,\ f(x,y) = ||y||^2-||x||^2$.

We denote $v:= \operatorname{grad} f$ its gradient vector field in the given coordinate system. For $(x,y)\in H_k$, we denote $\phi_t(x,y)\in H_k$ the point obtained by flowing along  $\frac{v}{v(f)}$ for a time $t$, when it is defined. We set:
$$\begin{array}{rcl}
    S^{k-1}\times S^{n+1-k}\times [-1,1]\times (\frac{1}{2},1] &\to& H_k\\
    (x,y,t,s) &\mapsto& \phi_{t+1}(\sqrt{s^2+1}\ x, sy)
\end{array}$$
This definition indeed gives a collar of the side boundary of $H_k$ as the function $g(x,y) = ||x||^2||y||^2$ is constant on flow lines of $f$. 

This example will guide the definition of a Morse function for a cobordism with corners in Definition \ref{Def_gradientlike}.
\end{definition}
\begin{figure}
    \centering
        \includegraphics[width=0.5\linewidth]{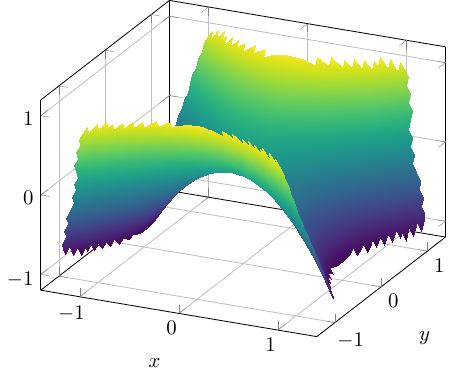}
\caption{A Morse saddle $H_1$ in dimension 2 (so n=0). The vertical (or side) boundary is made of gradient lines for $f(x,y) = y^2-x^2$. The top and bottom boundaries are made of level sets for $f$.}
    \label{fig:MorseSaddle}
\end{figure}
\begin{definition}\label{Def_hdlAttachment}
    Let $M: \Sigma_-\to \Sigma_+$ be a 1-morphism. A \textbf{framed $(k-1)$-sphere in $M$} is an embedding 
    \begin{equation*}
        \S: S^{k-1}\times \D^{n+2-k} \inj M
    \end{equation*}
    disjoint from the side collars. It induces a decomposition \begin{equation*}
        M \simeq (\id_{\Sigma_+}\sqcup\  S^{k-1}\times \D^{n+2-k}) \circ M_{\smallsetminus \S}
    \end{equation*}
    where $M_{\smallsetminus \S} := M\smallsetminus \S(S^{k-1}\times \D_{\frac{1}{2}}^{n+2-k}) : \Sigma_- \to \Sigma_+\sqcup S^{k-1}\times S^{n+1-k}$ with collars induced from those of $M$ and from $\S$.

The \textbf{manifold obtained from surgery on $M$ along $\S$} is the 1-morphism \begin{equation*}
    M(\S) := (\id_{\Sigma_+}\sqcup\ \D^k \times S^{n+1-k}) \circ M_{\smallsetminus \S} : \Sigma_- \to \Sigma_+\ .
\end{equation*} 

    We denote $a(\S) := \S(S^{k-1}\times \{0\}) \subseteq M$ and $b(\S) := \{0\}\times S^{n+1-k}\subseteq M(\S)$ called respectively the unframed attaching and belt spheres of the $k$-handle.

The \textbf{handle attachment on $M$ along $\S$} is the 2-morphism \begin{equation*}
    W(\S) := (\id \sqcup H_k) \circ_h \id_{M_{\smallsetminus \S}} : M \to M(\S)\ .
\end{equation*}
It is represented graphically in Figure \ref{fig:attachmentWofS}.
\end{definition}
\begin{figure}
    \centering
\begin{tikzpicture}[xscale = 2, yscale = 1.2]
\fill[gray!30, opacity = 0.5] (-0.8,0) rectangle ++(0.3,1.5);
\fill[gray!30, opacity = 0.5] (-0.7,0.5) rectangle ++(0.3,1.5);
\draw[very thick] (-0.8,0) --++(0,1.5);
\draw[very thick] (-0.7,0.5) --++(0,1.5);
\draw[very thick] (-0.8,0) -- (-0.5,0);
\draw[very thick] (-0.7,0.5) -- (-0.4,0.5);
\draw (-0.5,0) arc(-90:0:0.6 and 0.25) arc(0:90:0.5 and 0.25);
\draw (1,0) arc(-90:-180:0.5 and 0.25) arc(180:90:0.6 and 0.25)arc(90:0:0.5 and 0.25) arc(0:-90:0.6 and 0.25);
\node at (0,-0.2){$M$};
\begin{scope}[yshift = 1.5cm]
\draw[very thick] (-0.8,0) -- (-0.5,0);
\draw[very thick] (-0.7,0.5) -- (-0.4,0.5);
\draw (-0.5,0) to[out= 0, in=180] (0.25,0.1)to[out= 0, in=180] (1,0);
\draw (-0.4,0.5) to[out= 0, in=180] (0.35,0.4)to[out= 0, in=180] (1.1,0.5)arc(90:0:0.5 and 0.25) arc(0:-90:0.6 and 0.25);
\node at (0,0.7){$M(\S)$};  
\end{scope}
\fill[blue!40, opacity = 0.35] (-0.3,0.02) arc(-70:0:0.6 and 0.25) arc (180:0:0.2 and 0.75) arc(-180:-115:0.5 and 0.25)--++(0,1.5)  to[out= 180, in=0] (0.25,1.62) to[out= 180, in=0] (-0.3,1.52) -- cycle;
\fill[blue!40, opacity = 0.25] (-0.2,0.48) -- (-0.2,1.98) to[out= 0, in=180] (0.35,1.9)to[out= 0, in=180] (0.9,1.98) -- (0.9, 0.48) arc(110:180:0.6 and 0.25) arc (0:180:0.2 and 0.75) arc(0:70:0.5 and 0.25);
\draw[blue, very thick] (-2.1,0) --++ (0.5,0);
\draw[blue, very thick] (-2.1,0.5) --++ (0.5,0);
\node[scale = 1.4] at (-1.3, 0.35) {$\overset{\S}{\inj}$};
\draw[blue, very thick] (-0.3,0.02) arc(-70:0:0.6 and 0.25) arc(0:70:0.5 and 0.25);
\draw[blue, very thick] (0.8,0.02) arc(-115:-180:0.5 and 0.25) arc(180:110:0.6 and 0.25);
\draw[blue] (-0.3,1.52) to[out= 0, in=180] (0.25,1.62)to[out= 0, in=180] (0.8,1.52);
\draw[blue] (-0.2,1.98) to[out= 0, in=180] (0.35,1.9)to[out= 0, in=180] (0.9,1.98);
\draw[blue] (-0.3,0.02) --++(0,1.5);
\draw[blue] (0.8,0.02) --++(0,1.5);
\draw[blue] (-0.2,0.48) --++(0,1.5);
\draw[blue] (0.9,0.48) --++(0,1.5);
\draw[blue] (0.1,0.25) arc (180:0:0.2 and 0.75) node[midway, above]{$H_1$};
\draw (1.6,0.25) -- ++(0,1.5) node[midway, left]{$\id_{M\smallsetminus \S}$};
\end{tikzpicture}    \caption{The handle attachment along an embedded framed sphere for $n+2=2$ and $k=1$.}
    \label{fig:attachmentWofS}
\end{figure}

The above description of any handle attachment as a composition in the cobordism bicategory involving the standard generator $H_k$ makes the following algebraic definition possible.
\begin{definition}\label{Def_Z(S)}
    Let $\ZZ^\varepsilon: \Cob_{n+1+\varepsilon}\to \CC$ be a categorified TQFT and choose a 2-morphism $Z_k: \ZZ^\varepsilon(S^{k-1}\times \D^{n+2-k})\to \ZZ^\varepsilon(\D^k\times S^{n+1-k})$ in $\CC$, which we think of as the candidate assignment to the $k$-handle $H_k$.
    Given a framed $(k-1)$-sphere $\S \subseteq M$, we denote \begin{equation*}
        Z(\S) := (\id \otimes Z_k) \circ_h \id_{\ZZ^\varepsilon(M_{\smallsetminus \S})} : \ZZ^\varepsilon(M) \to \ZZ^\varepsilon(M(\S))
    \end{equation*} with the notations above, where we have left implicit the canonical isomorphisms $$\ZZ^\varepsilon(M) \simeq (\id_{\ZZ^\varepsilon(\Sigma_+)}\otimes  \ZZ^\varepsilon(S^{k-1}\times \D^{n+2-k})) \circ_h \ZZ^\varepsilon(M_{\smallsetminus \S})$$ and 
    $$(\id_{\ZZ^\varepsilon(\Sigma_+)}\otimes  \ZZ^\varepsilon(\D^{k}\times S^{n+1-k})) \circ_h \ZZ^\varepsilon(M_{\smallsetminus \S})\simeq \ZZ^\varepsilon(M(\S))\ .$$    
\end{definition}

\subsection{Handle cancellation}
\begin{definition}\label{Def_TheHdlCancel}
Let $B$ be a smoothing\footnote{One may embed the following manifold into $\D^{n+1}$ and take a neighborhood of its image. We can equip $B$ with a collar of its boundary that is disjoint from the image. Note that $B$ is diffeomorphic to an $(n+1)$-ball.} of $S^{k-1}\times \D^{n+2-k} \underset{S^{k-1}\times \D^{n+1-k}}{\cup} \D^k\times \D^{n+1-k}$ where $\D^{n+1-k} \inj \partial \D^{n+2-k}$ is identified with a hemisphere. It comes equipped with a framed $(k-1)$-sphere $\S_B: S^{k-1}\times \D^{n+2-k} \subseteq B$.

Now, $B(\S)$ is a smoothing of $\D^k\times S^{n+1-k} \underset{S^{k-1}\times \D^{n+1-k}}{\cup} \D^k\times \D^{n+1-k}$. We can find a framed $k$-sphere $\S_B'$ in $B(\S)$ via $\S_B': S^{k}\times \D^{n+1-k} \simeq \D^k\times \D^{n+1-k} \underset{S^{k-1}\times \D^{n+1-k}}{\cup} \D^k\times \D^{n+1-k} \subseteq B(\S)$.

It follows from \cite[Def. 2.17]{Juhasz} that $W(\S_B')\circ W(\S_B) \simeq W_{\varphi_B}$ for some diffeomorphism $$\varphi_B: B \to B(\S)(\S')$$ well defined up to isotopy, and supported in $S^{k-1}\times \D^{n+2-k} \underset{S^{k-1}\times \D^{n+1-k}}{\cup} \D^k\times \D^{n+1-k}$, hence preserving collars.
\end{definition}

\begin{definition}\label{Def_SatisfyHdlCancel}
    A pair of 2-morphisms
    $$Z_k: \ZZ^\varepsilon(S^{k-1}\times \D^{n+2-k})\to \ZZ^\varepsilon(\D^k\times S^{n+1-k})\ , \quad Z_{k+1}: \ZZ^\varepsilon(S^{k}\times \D^{n+1-k})\to \ZZ^\varepsilon(\D^{k+1}\times S^{n-k})$$  
    in $\CC$ is said to \textbf{satisfy the handle cancellation} if 
    \begin{equation*}
        Z(\S_B')\circ Z(\S_B) = \ZZ^\varepsilon(W_{\varphi_B})\ .
    \end{equation*} 
\end{definition}

\subsection{$\iota$-invariance}
Handle attachments have an inherent symmetry induced by reversing the attaching sphere. We need to ensure that our assignment to the handle attachments respects this symmetry.
\begin{definition}\label{Def_iota}
    For $1\leq k \leq n+1$, there is an orientation-preserving involution \begin{equation*}
        \begin{array}{rcl}
\iota: \Rr^k \times \Rr^{n+2-k} &\to& \Rr^k\times \Rr^{n+2-k}\\
        (x_1, x_2,\dots,x_k,y_1,y_2,\dots,y_{n+2-k}) &\mapsto& (-x_1,x_2,\dots,x_k,-y_1,y_2,\dots,y_{n+2-k})
    \end{array}
    \end{equation*} 
    which restricts to self-diffeomorphisms of $S^{k-1}\times S^{n+1-k},\ S^{k-1}\times \D^{n+2-k}, \ \D^{k}\times S^{n+1-k}$ and $H_k$. 

Using $\iota$, one can change any framed $(k-1)$-sphere $\S: S^{k-1}\times \D^{n+2-k}\inj M$ into another one $\overline{\S} := \S\circ \iota\vert_{S^{k-1}\times \D^{n+2-k}}$ which we call the reversed sphere (the orientation of its core $a(\S)$ has changed).

The surgered manifold $M(\overline{\S})$ hasn't changed and is diffeomorphic to $M(\S)$ using $\iota\vert_{\D^{k}\times S^{n+1-k}}$. From now on we will identity $M(\overline{\S})$ and $M(\S)$ using this canonical diffeomorphism.

The handle attachment $$W(\overline{\S}) = (\id \sqcup \overline{H_k})\circ M_{\smallsetminus \S}$$ is obtained by replacing $H_k$ with $\overline{H_k}$ which is $H_k$ as a manifold but where the source and target diffeomorphism and side collars have been twisted by $\iota$, i.e.
\begin{equation*}
   \overline{H_k} := W_{\iota\vert_{\D^{k}\times S^{n+1-k}}}\circ (H_k \circ_h \id_{M_{\iota\vert_{S^{k-1}\times S^{n+1-k}}}}) \circ W_{\iota\vert_{S^{k-1}\times \D^{n+2-k}}}
\end{equation*}
Finally, we have a diffeomorphism $\iota\vert_{H_k}: H_k \to \overline{H_k}$ preserving source and target diffeomorphism and side collars, i.e. inducing $\iota$ on the boundary. 

Hence, for any $\S\subseteq M$, the two 2-morphisms $W({\S})$ and $W(\overline{\S})$ are equal, where it is understood that their target has been identified using $\iota$ as above.



\end{definition}
\begin{definition}\label{Def_iotaInvariant}
    Let $\ZZ^\varepsilon: \Cob_{n+1+\varepsilon}\to \CC$ be a categorified TQFT. With the notations above, a 2-morphism $Z_k: \ZZ^\varepsilon(S^{k-1}\times \D^{n+2-k})\to \ZZ^\varepsilon(\D^k\times S^{n+1-k})$ in $\CC$ is called \textbf{$\iota$-invariant} if 
    \begin{equation*}
    \ZZ^\varepsilon(W_{\iota\vert_{\D^{k}\times S^{n+1-k}}})\circ (Z_k \circ_h \ZZ^\varepsilon(\id_{M_{\iota\vert_{S^{k-1}\times S^{n+1-k}}}})) \circ \ZZ^\varepsilon(W_{\iota\vert_{S^{k-1}\times \D^{n+2-k}}}) = Z_k
\end{equation*}
with appropriate compatibility isomorphisms of $\CC$ and $\ZZ^\varepsilon$ inserted, or equivalently if $Z(\S) = Z(\overline{\S})$ for any $\S\subseteq M$.
\end{definition}

\subsection{Main result}
\begin{theorem}\label{Thm_ETQFTfromHandle}
    Given a categorified TQFT $\ZZ^\varepsilon: \Cob_{n+1+\varepsilon}\to \CC$, $n\geq 0$, and 2-morphisms $$\left(Z_k: \ZZ^\varepsilon(S^{k-1}\times \D^{n+2-k})\to \ZZ^\varepsilon(\D^k\times S^{n+1-k})\right)_{k\in \{0,\dots,n+2\}}$$ such that \begin{itemize}
        \item[(a)] each pair $Z_k, Z_{k+1},\ k \in \{0,\dots,n+1\}$, satisfy the handle cancellation, and
        \item[(b)] each $Z_k,\ k \in \{1,\dots,n+1\}$, is $\iota$-invariant
    \end{itemize}
    there exists a unique $(n+1+1)$-TQFT \begin{equation*}
        \ZZ: \Cob_{n+1+1}\to \CC
    \end{equation*}
    extending $\ZZ^\varepsilon$ such that $\ZZ(H_k)=Z_k$.

    Reciprocally, given an $(n+1+1)$-TQFT $\ZZ$, taking $\ZZ^\varepsilon$ to be the restriction of $\ZZ$ to $\Cob_{n+1+\varepsilon}$, the morphisms $Z_k:= \ZZ(H_k)$ satisfy $(a)$ and $(b)$.

    In the non-compact case, given a categorified TQFT $\ZZ^\varepsilon$ and 2-morphisms $(Z_k)_{k\in \{0,\dots,n+1\}}$ satisfying (a) for $k\in \{0,\dots,n\}$ and (b), there exists a unique non-compact $(n+1+1)$-TQFT $\ZZ: \Cob_{n+1+1}^{nc}\to \CC$ extending $\ZZ^\varepsilon$ such that $\ZZ(H_k)=Z_k$ for $k\in \{0,\dots,n+1\}$, and reciprocally.
\end{theorem}
The next two sections will be devoted to the proof of this theorem.
We begin by proving a version of this statement for every Hom-category $\cob_{n+2}^{\Sigma_-,\Sigma_+}$ in Section \ref{Sec_MorseDatum}. Then we assemble these results together to have the bicategorical statement in Section \ref{Sec_2Fun}.

\section{Handle decompositions for manifolds with corners}\label{Sec_MorseDatum}

\subsection{A result on Hom categories}
The goal of this section is to prove a theorem analogous to \cite[Thm. 1.7]{Juhasz} for cobordisms with corners. The section is entirely based on \cite{Juhasz, GWW13, Milnor_hcobordism} which we adapt to cobordisms with collared corners. We begin by giving the statement.

\begin{definition}
    A \textbf{parametrized Cerf decomposition} for an $(n+2)$-cobordism with collared corners $W:M_-\to M_+$ is the data of:
\begin{itemize}
    \item collared $(n+1)$-cobordisms $M_1=M_-,M_2,\dots, M_{m+1}=M_+$ embedded in $W$
    \item a framed sphere $\S_i \subseteq M_i$ and a diffeomorphism $d_i:M_i(\S_i)\to M_{i+1}$ for each $1\leq i \leq m$. We also allow there to be no framed sphere, which \cite{Juhasz} calls the sphere $\S=\emptyset$ to smoothen notations, with $M(\emptyset):= M$ and $W(\emptyset):= \id_M$, in which case we have a diffeomorphism $d_i:M_i=M_i(\S_i)\to M_{i+1}$.
\end{itemize}
    such that there exists a diffeomorphism
    $$W \simeq W_{d_m}\circ W(\S_m)\circ \cdots\circ W_{d_1}\circ W(\S_1)$$
preserving the $M_i$'s pointwise and the side collars up to reparametrization (as in Definition \ref{Def_CobCorner}).

It is a \textbf{non-compact} decomposition if each $W_i$ is in $\Cob_{n+1+1}^{nc}$, i.e. if each $\S_i\subseteq M_i$ is a framed $(k-1)$-sphere for some $k<n+2$, or $\S_i=\emptyset$.

\end{definition}

\begin{definition}
    The \textbf{category of Cerf decompositions} $\operatorname{CD}_{n+2}^{\Sigma_-,\Sigma_+}$ has:
    \begin{description}
        \item[objects:] collared cobordisms $M: \Sigma_-\to \Sigma_+$
        \item[generating morphisms:] a morphism $e_{M,\S}: M \to M(\S)$ for every framed sphere $\S \inj M$, and\\
        a morphism $e_d: M\to M'$ for every diffeomorphism $d:M\to M'$ preserving collars.
        \item[relations:] the relation (1)--(5) of \cite{Juhasz}, namely:
        \begin{itemize}
            \item[(1)] $e_d \circ e_{d'} \sim e_{d\circ d'}$ whenever $d$ and $d'$ compose, and $e_d \sim \id$ whenever $d$ is isotopic to the identity.
            \item[(2)] $e_{M', d\circ \S}\circ e_d \sim e_{d^\S}\circ e_{M,\S}$ for $d:M\to M'$ and $\S\subseteq M$, where $d^\S: M(\S) \to M'(d\circ \S)$ is the diffeomorphism induced by $d$.
            \item[(3)]  $e_{M(\S),\S'}\circ e_{M,\S}\sim e_{M(\S'),\S}\circ e_{M,\S'}$ whenever $\S,\S' \subseteq M$ are disjoint, so $\S'$ can be pushed to $M(\S)$ and $\S$ to $M(\S')$.
            \item[(4)]  $e_{M(\S),\S'}\circ e_{M,\S} \sim e_\varphi$ whenever $a(\S')$ and $b(\S)$ intersect transversely exactly once in $M(\S)$, and $\varphi$ is the induced diffeomorphism $M\simeq M(\S)(\S')$. 
            \item[(5)] $e_{M,\S} \sim e_{M,\overline{\S}}$ where $\S$ is a framed $(k-1)$-sphere for some $1\leq k\leq n+1$ and $\overline{\S} = \S\circ \iota$, and $M(\overline{\S})$ has been canonically identified with $M(\S)$ as in Definition \ref{Def_iota}.
        \end{itemize}
    \end{description}
    
    The \textbf{category of non-compact Cerf decompositions} $\operatorname{CD}_{n+2}^{\Sigma_-,\Sigma_+, nc}$ has the same objects, same generating morphisms except the morphisms $e_{M,\S}$ where $\S\subseteq M$ is a framed $(n+1)$-sphere, and the same relations except those involving the discarded generating morphisms.
\end{definition}
The following result is the direct analogue of \cite[Thm 1.7]{Juhasz} with corners.
\begin{theorem}\label{Thm_JuhaszCorner}
    Let $\Sigma_\pm$ be closed oriented smooth $n$-manifolds. Then there is an equivalence of categories 
    \begin{equation*}
        \operatorname{CD}_{n+2}^{\Sigma_-,\Sigma_+} \simeq \cob_{n+2}^{\Sigma_-,\Sigma_+}
    \end{equation*}
    which is the identity on objects and maps $e_{M,\S}$ to $W(\S)$ and $e_d$ to $W_d$.

    It restricts to an equivalence $\operatorname{CD}_{n+2}^{\Sigma_-,\Sigma_+, nc} \simeq \cob_{n+2}^{\Sigma_-,\Sigma_+, nc}$.
\end{theorem}
\begin{corollary}\label{Cor_Juhasz}
    Given a functor \begin{equation*}
        F^\varepsilon: \cob_{(n+1)+\varepsilon}^{\Sigma_-,\Sigma_+} \to \CC
    \end{equation*}
    and a morphism $F_{M,\S}: F(M) \to F(M(\S))$ for every $\S \subseteq M$ (resp. every $\S$ of index $k<n+1$) satisfying relations (2)--(5), there exists a unique functor \begin{equation*}
        F: \cob_{n+2}^{\Sigma_-,\Sigma_+}\to \CC \quad\quad (\text{resp. } F: \cob_{n+2}^{\Sigma_-,\Sigma_+, nc}\to \CC)
    \end{equation*}
    extending $F^\varepsilon$ and such that $F(W(\S))=F_{M,\S}$.
\end{corollary}

The proof follows \cite{Juhasz}. In essence, we choose a Morse function and a gradient-like vector field on our cobordisms, and check that this gives a parametrized Cerf decomposition which does not depend on the choices of Morse function and gradient-like vector field up to relations (1)--(5).

The only difference with \cite{Juhasz} is the presence of corners which might interact with the handle attachments. This problem is avoided by asking that gradient-like vector fields are vertical in collars. Hence flowing along a gradient-like vector field will never make some point of the interior of the manifold enter the corners, see Remark \ref{rmk:flowverticalside}.

\subsection{Morse data and decompositions}

\begin{definition}\label{Def_gradientlike}
    Let $\Sigma_\pm$ be $n$-manifolds, $M_\pm: \Sigma_-\to \Sigma_+$ be collared $(n+1)$-cobordisms and $W: M_-\to M_+$ a collared $(n+2)$-cobordism. A \textbf{Morse function} on $W$ is a smooth function
\begin{equation*}
    f: W \to [-1,1]
\end{equation*}
such that $f$ has finitely many critical points, all non-degenerate and happening away from the boundary and the side collars, and such that $f$ restricted to the side collars $\Sigma\pm\times[-1,1]\times[\pm1,\pm\frac{1}{2})\inj W$ coincides with projection on the $[-1,1]$-coordinate, and the source and target diffeomorphisms map $M_\pm$ isomorphically onto $f^{-1}(\pm1)$. It is called \textbf{excellent} if all the critical points of $f$ have distinct values.

A \textbf{gradient-like vector field} for $f$ is a vector field $v$ on $W$ such that for every $p \in W \smallsetminus \operatorname{Crit}(f),\ v(f)>0$ and for every $p \in \operatorname{Crit}(f)$ there exists an orientation-preserving coordinate system $(x_1,\dots,x_k, y_1,\dots, y_{n+2-k})$ centered at $p$ such that $f(x,y) = f(p)+ ||y||^2-||x||^2$ and $v = \operatorname{grad}f$. On the side collars $\Sigma_\pm\times [\pm1,\pm\frac{1}{2})\times [-1,1]$, we require that $v=\partial t$ where $t$ is the $[-1,1]$-coordinate.

A \textbf{Morse datum} for $W$ is the data of 
\begin{itemize}
    \item an excellent Morse function $f:W\to [-1,1]$
    \item a tuple $\underline{b}= \{ b_0=-1<b_1<\cdots<b_{m+1}=1\}$ such that $b_i$ is a regular value for $f$ and $f$ has at most one critical point in $f^{-1}([b_i, b_{i+1}])$
    \item a gradient-like vector field $v$ for $f$
\end{itemize}
See Figure \ref{fig:MorseDatum}.

Let $\Phi: W\to W'$ be a diffeomorphism preserving $M_\pm$ and preserving the side collars up to the reparametrization $\varphi:[-1,1]\to[-1,1]$ as in Definition \ref{Def_CobCorner}. A Morse datum $(f,v,\underline{b})$ on $W$ induces a Morse datum $(\varphi\circ f \circ \Phi^{-1}, \varphi'(f(-))\cdot\Phi^{-1}_*v,  \varphi(\underline{b}))$ on $W'$.
\end{definition}
\begin{figure}
    \centering
\begin{tikzpicture}[xscale = 1.5]
\draw[very thick] (-0.9,0.5) --++(0.5,0);
\draw (-0.5,0) arc(-90:0:0.3 and 0.25) arc(0:90:0.2 and 0.25);
\draw (1,0)--(0.5,0) arc(-90:-180:0.2 and 0.25) arc(180:90:0.3 and 0.25);
\draw[dashed] (0.6,0.5) --++(0.5,0);
\draw (-0.9,0.5)  -- ++(0,4);
\fill[gray!30, opacity = 0.5] (-1,0) rectangle ++(0.5,4);
\fill[gray!30, opacity = 0.5] (-0.9,0.5) rectangle ++(0.5,4);
\draw (-1,0) -- ++(0,4);
\draw[very thick] (-1,0) -- ++(0.5,0);
\draw[very thick] (-1,4) --++(0.5,0);
\draw[very thick] (-0.9,4.5) --++(0.5,0);
\draw (-0.9,4.5) --++(3.5,0);
\draw (-1,4) --++(3.5,0);
\draw (-0.2,0.25) arc(180:0:0.25 and 1);
\node[scale = 0.5, rotate = 80] at (-0.25,0.8){$\to$};
\node[scale = 0.5, rotate = 100] at (0.35,0.8){$\to$};
\node[scale = 0.5, rotate = 60] at (-0.15,1.15){$\to$};
\node[scale = 0.5, rotate = 120] at (0.25,1.15){$\to$};
\node[scale = 0.5, rotate = 60, black!50] at (0.13,1.45){$\to$};
\node[scale = 0.5, rotate = 120] at (-0.03,1.45){$\to$};
\node[scale = 0.5, rotate = 75, black!50] at (0.25,1.8){$\to$};
\node[scale = 0.5, rotate = 105] at (-0.15,1.8){$\to$};
\node[scale = 0.5, rotate = 90] at (-0.75,0.3){$\to$};
\node[scale = 0.5, rotate = 90] at (-0.75,1.15){$\to$};
\node[scale = 0.5, rotate = 90] at (-0.75,2){$\to$};
\node[scale = 0.5, rotate = 90] at (-0.75,2.85){$\to$};
\node[scale = 0.5, rotate = 90] at (-0.75,3.7){$\to$};
\begin{scope}[xshift = 0.1cm, yshift = 0.5cm]
\node[scale = 0.5, rotate = 90, black!50] at (-0.75,0.3){$\to$};
\node[scale = 0.5, rotate = 90, black!50] at (-0.75,1.15){$\to$};
\node[scale = 0.5, rotate = 90, black!50] at (-0.75,2){$\to$};
\node[scale = 0.5, rotate = 90, black!50] at (-0.75,2.85){$\to$};
\node[scale = 0.5, rotate = 90] at (-0.75,3.7){$\to$};
\end{scope}
\node[scale = 0.5, rotate = 90] at (0.75,1){$\to$};
\node[scale = 0.5, rotate = 90] at (0.75,2){$\to$};
\node[scale = 0.5, rotate = 90] at (0.75,3){$\to$};
\node[scale = 0.5, rotate = 90, black!50] at (0.85,1.5){$\to$};
\node[scale = 0.5, rotate = 90, black!50] at (0.85,2.5){$\to$};
\node[scale = 0.5, rotate = 90, black!50] at (0.85,3.5){$\to$};
\begin{scope}[xshift = 1.5cm, yshift = 2cm]
\draw[very thick] (0.6,-1.5) -- ++(0.5,0);
\draw (-0.5,-2) arc(-90:0:0.3 and 0.25) arc(0:90:0.2 and 0.25);
\draw (0.5,-2) arc(-90:-180:0.2 and 0.25) arc(180:90:0.3 and 0.25);
\draw (-1,2) --++(2,0);
\draw (-0.9,2.5) --++(2,0);
\fill[gray!30, opacity = 0.5] (0.5,-2) rectangle ++(0.5,4);
\fill[gray!30, opacity = 0.5] (0.6,-1.5) rectangle ++(0.5,4);
\draw (1,-2) -- ++(0,4);
\draw (1.1,-1.5) -- ++(0,4);
\draw[very thick] (0.5,-2) -- ++(0.5,0);
\draw[very thick] (0.5,2) -- (1,2);
\draw[very thick] (0.6,2.5) -- (1.1,2.5);
\draw (-0.2,0.25) arc(180:0:0.25 and 1);
\draw (-0.2,0.25) -- ++(0,-2);
\draw (0.3,0.25) -- ++(0,-2);
\node[scale = 0.5, rotate = 80] at (-0.25,0.8){$\to$};
\node[scale = 0.5, rotate = 100] at (0.35,0.8){$\to$};
\node[scale = 0.5, rotate = 60] at (-0.15,1.15){$\to$};
\node[scale = 0.5, rotate = 120] at (0.25,1.15){$\to$};
\node[scale = 0.5, rotate = 60, black!50] at (0.13,1.45){$\to$};
\node[scale = 0.5, rotate = 120] at (-0.03,1.45){$\to$};
\node[scale = 0.5, rotate = 75, black!50] at (0.25,1.8){$\to$};
\node[scale = 0.5, rotate = 105] at (-0.15,1.8){$\to$};
\end{scope}
\begin{scope}[xshift = 3cm, yshift = 0cm]
\node[scale = 0.5, rotate = 90] at (-0.75,0.3){$\to$};
\node[scale = 0.5, rotate = 90] at (-0.75,1.15){$\to$};
\node[scale = 0.5, rotate = 90] at (-0.75,2){$\to$};
\node[scale = 0.5, rotate = 90] at (-0.75,2.85){$\to$};
\node[scale = 0.5, rotate = 90] at (-0.75,3.7){$\to$};
\end{scope}
\begin{scope}[xshift = 3.1cm, yshift = 0.5cm]
\node[scale = 0.5, rotate = 90, black!50] at (-0.75,0.3){$\to$};
\node[scale = 0.5, rotate = 90, black!50] at (-0.75,1.15){$\to$};
\node[scale = 0.5, rotate = 90, black!50] at (-0.75,2){$\to$};
\node[scale = 0.5, rotate = 90, black!50] at (-0.75,2.85){$\to$};
\node[scale = 0.5, rotate = 90] at (-0.75,3.7){$\to$};
\end{scope}
\draw[very thick] (4,-0.25)--++(0,0.5) node{$-$} node[right]{$b_0=-1$}-- ++(0,2) node{$-$} node[right]{$b_1=0$}-- ++(0,2) node{$-$} node[right]{$b_2=1$} --++(0,0.5);
\node[scale = 1.3] at (3.25,2.25) {$\overset{f}{\longrightarrow}$};
\end{tikzpicture}    
\caption{A Morse datum on a cobordism with corners.}
    \label{fig:MorseDatum}
\end{figure}
\begin{remark}\label{rmk:flowverticalside}
    We can define a vector field $\frac{v}{v(f)}$ away from critical points. Let us denote $\phi_t(x)$ the flow along this vector field from the point $x\in W$ for a time $t\in\mathbb R$, when it is defined. Then $f(\phi_t(x)) = f(x)+t$. Because $v$ is vertical in the side corners, this flow cannot transport a point of the interior to the side boundary. It is therefore defined on the largest interval $[a,b]$ such that $f(\phi_a(x))=-1$ or $\phi_a(x)$ is a critical point, and $f(\phi_b(x))=1$ or $\phi_b(x)$ is a critical point.
\end{remark}
\begin{remark}
In the cobordism bicategory $\Cob_{n+1+1}$, there are two directions to compose in, which we called horizontal and vertical. In this section, we are interested in decomposing 2-morphisms as a vertical composition. Therefore, our Morse functions and vector fields are vertical.
\end{remark}
\begin{example}
    If $d:M_-\to M_+$ is a diffeomorphism and $W_d=M_+\times [-1,1]$ is the induced 2-morphism, then the projection on the $[-1,1]$-coordinate $t$ is a Morse function, without any critical point, and $v=\partial t$ is a gradient-like vector field. A Morse datum is given by taking $\underline{b}=\{-1,1\}$
\end{example}
\begin{example}  
    The cobordism $H_k$ from Definition \ref{Def_handle} is equipped with a Morse function $f_k = ||y||^2-||x||^2$ in the sense above. It comes with a coordinate system on all of $H_k$ such that $v=\operatorname{grad} f$ is a gradient-like vector field in the sense above except that it may be a positive scalar times $\partial t$ on the side collars. We may renormalize $v$ away from the critical point so that it is strictly equal to $\partial t$ on the side collars.
    
    If $M$ is a 1-morphism and $\S \subseteq M$ a framed sphere, then $W(\S) = (\id\sqcup H_k)\circ_h \id_{M'}$ also has a Morse function which is the projection on the $[-1,1]$-coordinate on the $\id$ components and $f_k$ on $H_k$. They glue because they are both the $[-1,1]$-coordinate on side collars. The gradient-like vector fields glue similarly into a gradient-like vector field on  $W(\S)$. A Morse datum is given by taking $\underline{b}=\{-1,1\}$
\end{example}
A 2-morphism $W$ equipped with a Morse datum with $\underline{b}=\{-1,1\}$ is called \textbf{elementary}.
\begin{lemma}\label{Lem_elementaryCob}
    Any elementary cobordism $W:M_-\to M_+$ is of the form $W= W_d \circ W(\S)$ for some framed sphere $\S\subseteq M_-$, possibly $\S=\emptyset$ i.e. $W(\S)=\id_M$, and diffeomorphism $d: M_-(\S)\to M_+$. Moreover, these are well-defined up to isotopy of $d$ and $\S$ and reversal $\S\leftrightarrow\overline{\S}$.
\end{lemma}
\begin{proof}
    If the Morse function $f$ on $W$ doesn't have any critical point, then we take $\S=\emptyset$ and the diffeomorphism $d$ is given by flowing along $\frac{v}{v(f)}$ from $M_-$ to $M_+$ as in \cite[Thm. 3.4]{Milnor_hcobordism}. This flow is well defined away from $M_\pm$ by Remark \ref{rmk:flowverticalside}. 

    Otherwise, $f$ has a single critical point $p\in W$ of index $k$. There exists an orientation-preserving coordinate system $(x_1,\dots,x_k, y_1,\dots, y_{n+2-k})$ centered at $p$ such that $f(x,y) = f(p)+ ||y||^2-||x||^2$ and $v=\operatorname{grad}f$. Let $\varepsilon>0$ be small enough so that $\varepsilon H_k \subseteq \Rr^{n+2}$ is contained in the coordinate system, hence can be identified with a subspace of $W$. Note that by construction the image of the side boundary of $\varepsilon H_k$ is made of gradient lines for $f$, and the source and target boundary lie in the level sets $f^{-1}(f(p)\pm \varepsilon^2)$. Away from $\varepsilon H_k$, $f$ has no critical points and we can flow along $\frac{v}{v(f)}$. We still denote $\phi_t(x)\in W$ the point obtained by flowing for a time $t$ starting from $x$. 

    The framed $(k-1)$-sphere is given by $$\S: S^{k-1}\times \D^{n+2-k} \overset{\varepsilon\cdot}{\longrightarrow}
    f^{-1}(f(p)- \varepsilon^2) \overset{\phi_{-1-f(p)+\varepsilon^2}}{\longrightarrow}M_-$$
and more generally flowing from the source boundary $\varepsilon S^{k-1}\times \D^{n+2-k} \subseteq \varepsilon H_k \subseteq W$ for a time $t-f(p)+\varepsilon^2$ gives an embedding $S^{k-1}\times \D^{n+2-k} \times [-1, f(p)-\varepsilon^2]\inj W$ on which $f$ restricts to the projection on the $t$-coordinate. 

Similarly, We get an embedding $\D^{k}\times S^{n+1-k} \times [f(p)+\varepsilon^2,1] \inj W$. Denote $$H = \big(\D^{k}\times S^{n+1-k} \times [f(p)+\varepsilon^2,1]\big) \cup\varepsilon H_k \cup \big(S^{k-1}\times \D^{n+2-k} \times [-1, f(p)-\varepsilon^2]\big) \subseteq W$$
which is diffeomorphic to $\id \circ H_k \circ \id \simeq H_k$.

Away from $H$, the flow $\phi$ gives an identification 
$$W\smallsetminus H \simeq (M_-\smallsetminus \S)\times [0,1]$$
and hence a diffeomorphism $f^{-1}(1) \simeq M_-(\S)$. Along with the target diffeomorphism, this gives us the desired diffeomorphism $d: M_-(\S)\simeq M_+$.
We have decomposed $W$ into 
$$W \simeq (W\smallsetminus H) \cup H \simeq (M_-\times [-1,1])\cup H_k$$
with target diffeomorphism $d$, i.e. $W \simeq W_d \circ W(\S)$ as claimed.

As in \cite[Rem. 3.13]{Juhasz}, this construction only depends on the choice of the coordinate system and the value of $\varepsilon$. Isotopic coordinate systems and $\varepsilon$ will give isotopic $\S$ and $d$. The space of such coordinate systems is homotopy equivalent to $SO(k, n+2-k)$, which is connected for $k\in\{0,n+2\}$ and has two components, corresponding to $\S$ and $\overline{\S}$ otherwise.
\end{proof}
\begin{proposition}
    Let $W$ be a collared $(n+2)$-cobordism. A Morse datum $(f,v,\underline{b})$ on $W$ induces a parametrized Cerf decompositions of $W$ well-defined up to relation (1), (2) and (5).
\end{proposition}
\begin{proof}
     Each $W_i = f^{-1}([b_i, b_{i+1}])$ is an elementary cobordism from $M_i=f^{-1}(b_i)$ to $ M_{i+1}=f^{-1}(b_{i+1})$. By Lemma \ref{Lem_elementaryCob}, it is diffeomorphic to a composition of $W_i \simeq W_{d_i}\circ W(\S_i)$ where $d_i$ is well-defined up to isotopy, and $\S_i$ up to isotopy and reversal $\S_i\leftrightarrow\overline{\S}_i$, all of which are implied by relations (1), (2) and (5).
\end{proof}
The following Lemma is analogous to \cite[Lem. 3.7]{Milnor_hcobordism}, see also \cite[Lemma 2.6]{GWW13}.
\begin{lemma}\label{Lem_functionsGlue}
    Given $W:M_1\to M_2$ and $W':M_2\to M_3$ both equipped with Morse datum $(f,v,\underline{b})$ and $(f',v', \underline{b}')$, then there is a canonical smooth structure on the composition $W'\circ W$, and a canonical Morse datum well-defined up to isotopy of $v$.
\end{lemma}
\begin{proof}
    Since $f$ has no critical points near the boundary, flowing along $\frac{v}{v(f)}$ gives a collar for $M_2$ in $W$ in which $f$ is the projection on the $t$-coordinate. Similarly, flowing along $\frac{v'}{v'(f')}$ gives a collar for $M_2$ in $W'$ in which $f'$ is the projection on the $t$-coordinate. We may use these collars to define the smooth structure on $W'\circ W$, which ensures that $f-1$ and $f'+1$ glue into a smooth function on $W'\circ W$. The gradient-like vector fields are both vertical in the collars, but may not be normalized. We can renormalize them so that they agree strictly with $\partial t$ near $M_2$, and are unchanged away from its collars. Hence they glue into a smooth vector field. 
    
    The triple $(r((f-1)\cup(f'+1)), r'\cdot(v\cup v'), r((\underline{b}-1)\cup( \underline{b}'+1))$ is a Morse datum for $W'\circ W$, where $r:[-2,2]\to [-1,1]$ is the reparametrization used for composition in Definition \ref{Def_CobCorner}. 
\end{proof}
We get a version of \cite[Lem. 2.15]{Juhasz} for corners:
\begin{proposition}\label{Prop_CDtoMorse}
    Let $W$ be a collared $(n+2)$-cobordism and pick a parametrized Cerf decomposition of $W$. There exists a Morse datum for $W$ inducing this decomposition up to relations (1), (2) and (5).
\end{proposition}
\begin{proof}    
    We have seen that the elementary cobordisms of the form $W_d$ or $W(\S)$ carry Morse data. By Lemma \ref{Lem_functionsGlue}, any composition of elementary cobordisms also carries a Morse datum. 

    Our parametrized Cerf decomposition tells us that $W$ is diffeomorphic to a prescribed composition of elementary cobordisms, whose Morse datum transports to a Morse datum on $W$ under the diffeomorphism. By construction, this Morse datum induces the given decomposition of $W$.   
\end{proof}

\subsection{Existence and unicity of Morse data}
\begin{lemma}\label{Lem_ExistCD}
    Let $W$ be a collared $(n+2)$-cobordism. Then $W$ admits a Morse datum, and hence a parametrized Cerf decomposition. 
\end{lemma}
\begin{proof}
We follow \cite{Milnor_hcobordism}.
Choose a collar $M_\pm\times [\pm1,0) \subseteq W$, and an extended collar $\Sigma_\pm\times[-1,1]\times[\pm1,0)\subseteq  W$. One can cover $W$ by $(U_\Sigma,U_M,U_W)$ where $U_\Sigma = \Sigma_\pm\times[-1,1]\times[\pm1,0)$, $U_M = \big(M_\pm\times [\pm1,0)\big) \smallsetminus \big(\Sigma_\pm\times[-1,1]\times[\pm1,\pm\frac{1}{2}]\big) $ and $U_W= W \smallsetminus \big(M_\pm\times [\pm1,\pm\frac{1}{2}]\cup \Sigma_\pm\times[-1,1]\times[\pm1,\frac{1}{2}]\big)$. Choose a partition of unity $\rho_\Sigma,\rho_M,\rho_W$ subordinate to this cover. The projection on the $[-1,1]$-coordinate is a smooth function $f_\Sigma$ on $U_\Sigma$, and $\rho_\Sigma\cdot f_\Sigma$ extends to a smooth function on $W$ which is zero outside $U_\Sigma$. Similarly, the projection on the $[\pm 1,0)$-coordinate is a smooth function $f_M$ on $U_M$, and $\rho_M\cdot f_M$ extends to a smooth function on $W$ which is zero outside $U_M$.

For generic function $f_W:U_W\to (-\delta,\delta)$, $\delta>0$ sufficiently small, the function $$f= \rho_\Sigma f_\Sigma+ \rho_M f_M + \rho_W f_W$$ is excellent Morse and is the projection on the $t$-coordinate on the side collar. It has no critical points near $M_\pm=f^{-1}(\pm1)$ as both $f_\Sigma$ and $f_M$ must have positive derivative in the $t$ direction for any collar of $M_\pm$. 

It admits a gradient-like vector field $v$ by similar arguments. One covers $\overline{U_W}$ by coordinate systems $(U_i)_i$ on which $f$ is either a standard Morse singularity or is the projection on the first coordinate as in \cite[Lemma 3.2]{Milnor_hcobordism}. One consider the vector field $v_i$ given the gradient of $f$ in this coordinate system, and $\partial_t$ on $U_\Sigma$. One chooses a partition of unity $\rho_\Sigma, (\rho_i)_i$ of $W$ subordinate to the cover $U_\Sigma, (U_i\cap U_W)_i$ and sets $v = \rho_\Sigma \partial_t + \sum_i \rho_i v_i$.

By choosing at least one point $b_i$ in between each critical value of $f$, we have obtained a Morse datum $(f,v,\underline{b})$.
\end{proof}

\begin{proposition}\label{Prop_FamillyMorse}
    Let $(f,v,\underline{b})$ and $(f',v',\underline{b}')$ be two Morse data on a collared $(n+2)$-cobordism $W$. Suppose there exists a smooth path of function $(f_\lambda)_{\lambda\in[0,1]}$ between $f = f_0$ and $f'=f_1$ where each $f_\lambda$ is an excellent Morse function on $W$, and smooth paths  $(v_\lambda)_\lambda$ of gradient-like vector fields and $(\underline{b}_\lambda)_\lambda$ of regular values. Then the parametrized Cerf decompositions induced by $(f,v,\underline{b})$ and $(f',v',\underline{b}')$ are related by relations (1), (2) and (5).
\end{proposition}
\begin{proof}
By \cite[Lemma 3.1]{GWW13} there exists a diffeormorphism $\psi:W\to W$ intertwining the parametrized Cerf decompositions induced by $(f_0,v_0, \underline{b}_0)$ and $(f_1,v_1, \underline{b}_1)$. We recall their construction as we need to check that this diffeomorphism preserves collars up to reparametrization. 

Each Morse datum $(f_\lambda,v_\lambda, \underline{b}_\lambda)$ induces a bicollar $M_{i, \lambda}    \times [f_\lambda(b_{i, \lambda})-\varepsilon,f_\lambda(b_{i, \lambda})+\varepsilon]\inj W$ of every level set $M_{i, \lambda}:= f^{-1}(b_i)$ by flowing along $\frac{v_\lambda}{v_\lambda(f_\lambda)}$. These bicollars vary smoothly in $\lambda$. Note that as $v_\lambda=\partial t$ on the side collar, these bicollars are compatible with the side collars. Let $h: \Rr\to \Rr$ be a smooth bump function which is constant equal to 1 near 0 and has support in $[-\varepsilon,\varepsilon]$ and define $h_{i, \lambda}:W\to \Rr$ with $h_{i, \lambda}(x,t) := h(t-f_\lambda(b_{i, \lambda}))$ for $(x,t)\in M_{i, \lambda} \times [f_\lambda(b_{i, \lambda})-\varepsilon,f_\lambda(b_{i, \lambda})+\varepsilon]$ and 0 elsewhere.

We consider the vector field $V$ on $[0,1]\times W$ defined by 
$$V(\lambda,x) = \bigg(\partial_\lambda, h_{i, \lambda}(\lambda,x)\big(\partial_\lambda b_{i, \lambda} - \partial_\lambda f_\lambda(x)\big) \frac{v_\lambda}{v_\lambda(f_\lambda)}\bigg)$$
The diffeomorphism $\psi:W\to W$ is given by flowing along $V$ from a point $(0,x)$ to a point $(1, \psi(x))$.

On a point $x = (p,t,s) \in \Sigma_\pm\times [-1,1]\times [\pm1, \pm\frac{1}{2})$, the functions $h_{i, \lambda}(x)$ and $f_\lambda(x)$ depend only on $t$, and the vector field $v_\lambda$ is vertical. Hence, $\psi(p,t,s) = (p, \varphi(t),s)$ affects only the $t$-coordinate, uniformly in $p$ and $s$, and $\psi$ preserves the side collars up to reparametrization.

By \cite[Lemma 3.1 (b)]{GWW13} the Morse datum $(f',v',\underline{b}')$ is obtained, up to isotopy of $v'$, by transporting $(f,v,\underline{b})$ under the diffeomorphism $\psi$ with reparametrization $\varphi$. Note that they alter the vector field $V$ and the Morse data to obtain this, but it is again straightforward that this alteration preserves collars up to reparametrization. It is checked that the induced parametrized Cerf decompositions are related by relations (1), (2) and (5) in \cite[page 27]{Juhasz}. 
\end{proof}

\begin{proposition}\label{Prop_UniquenessCD}
    Let $W$ be a collared $(n+2)$-cobordism. Then any two parametrized Cerf decompositions of $W$ are related by relations (1)--(5).
\end{proposition}
\begin{proof}
First, by Proposition \ref{Prop_CDtoMorse}, any parametrized Cerf decomposition is induced by a Morse datum, up to relations (1), (2), (5), hence it is enough to consider moves between Morse data. 

Suppose that $(f,v,\underline{b})$ and $(f',v',\underline{b}')$ are two Morse data on $W$. If $f=f'$, then $v$ and $v'$ are isotopic, see \cite[Def 2.10]{Juhasz}. If $f=f'$ and $v=v'$, then $(f,v,\underline{b}\cup \underline{b}')$ is a Morse datum and $(f,v,\underline{b})$ and $(f',v',\underline{b}')$ can both be obtained from it by removing points of $\underline{b}\cup \underline{b}'$. In both cases, this does not affect the induced parametrized Cerf decomposition up to relations (1), (2), (5), see \cite[Lem. 2.21 and 2.22]{Juhasz}.


If $f \neq f'$, there exists a smooth path $\overline f_\lambda=\lambda f+(1-\lambda )f',\ \lambda\in[0,1]$ between these two functions. This path may have very degenerate singularities in general, however note that it is still the projection on the $t$-coordinate on the side collar and has no critical points near $M_\pm=\overline f_\lambda^{-1}(\pm1)$. By compactness, all $\overline f_\lambda$'s have no critical point on $U_\partial= \big(M_\pm\times [\pm1,\pm1\mp 2\varepsilon)\big)\cup\big( \Sigma_\pm\times[-1,1]\times[\pm1,\frac{1}{2}\mp 2\varepsilon)\big)$ for some $\varepsilon>0$. Set $U_{in} = W \smallsetminus \big(M_\pm\times [\pm1,\pm1\mp \varepsilon]\cup \Sigma_\pm\times[-1,1]\times[\pm1,\frac{1}{2}\mp \varepsilon]\big)$ and choose a partition of unity $\rho_\partial,\rho_{in}$ adapted to the cover $(U_\partial, U_{in})$. By Cerf theory, for generic paths of functions $f_{in,\lambda}:U_{in}\to  (-\delta,\delta)$, $\delta>0$ sufficiently small, the path $$f_\lambda = \rho_{in} f_{in,\lambda} + \overline f_\lambda$$ is a path of excellent Morse function for $W$ except at finitely many times $\lambda_i\in (0,1)$ where it has either a birth/death singularity or a critical value crossing.
If there are no singularities, i.e. $(f_\lambda)_\lambda$ consists only of excellent Morse functions, then we can extend $\underline{b}$ into a path of regular values $\underline{b}_\lambda$ and $v$ to a path of gradient-like vector fields $v_\lambda$.

As in \cite[Thm. 2.24]{Juhasz}, one can modify this path slightly in a neighborhood of the singularity (hence away from the boundary) so that $(f_\lambda)_{\lambda\in (\lambda_i-\varepsilon,\lambda_i+\varepsilon)}$ is a standard ``chemin élémentaire de mort" \cite[p.71, Prop. 2]{Cerf} or a standard ``chemin élémentaire de 1-croisement" \cite[p.49, Prop. 2]{Cerf}. In particular, it is constant near the boundary and side collars. The parametrized Cerf decompositions induced by $f_{\lambda_i-\varepsilon}$ and $f_{\lambda_i+\varepsilon}$ are related by relations (1)--(5) by the arguments of \cite[Lem. 2.19 and 2.20]{Juhasz}. Finally, the parametrized Cerf decompositions induced by $f_{\lambda_i+\varepsilon}$ and $f_{\lambda_{i+1}-\varepsilon}$ are related by relations (1), (2) and (5) by Proposition \ref{Prop_FamillyMorse}.
\end{proof}

\begin{proof}[Proof of Theorem \ref{Thm_JuhaszCorner}]
The proof of \cite[Thm. 1.7]{Juhasz} now applies similarly. There is a functor $\operatorname{CD}_{n+2}^{\Sigma_-,\Sigma_+} \to \cob_{n+2}^{\Sigma_-,\Sigma_+}$ which is the identity on objects, maps $e_d$ to $W_d$ and $e_{M,\S}$ to $W(\S)$. It is well-defined as relations (1)--(5) do hold on the cobordisms. It is essentially surjective as it is the identity on objects. It is full by Lemma \ref{Lem_ExistCD} and faithful by Proposition \ref{Prop_UniquenessCD}.

In the non-compact case, by the argument of \cite[Thm 2.24]{Juhasz}, \cite{KirbyLinksS3} one can suppose that the path $(f_\lambda)_\lambda$ in Proposition \ref{Prop_UniquenessCD} does not introduce $(n+2)$-handles.
\end{proof}

\section{Inducing a 2-functor from handle attachment data}\label{Sec_2Fun}
We now finish the proof of Theorem \ref{Thm_ETQFTfromHandle}. We suppose we are given a categorified TQFT $\ZZ^\varepsilon$ and 2-morphisms $\big(Z_k: \ZZ^\varepsilon(S^{k-1}\times \D^{n+2-k})\to \ZZ^\varepsilon(\D^k\times S^{n+1-k})\big)_{k\in \{0,\dots,n+2\}}$ (resp. $k\in \{0,\dots,n+1\}$ in the non-compact case) satisfying (a) and (b).
We want to build the $(n+1+1)$-TQFT $$\ZZ: \Cob_{n+1+1}\to \CC \quad \text{  (resp. }\ZZ: \Cob_{n+1+1}^{nc}\to \CC \text{  in the non-compact case).}$$ 

Let us summarize the strategy in a few words.
A 2-functor $\ZZ$ is the data of an assignment on objects, which is already given by $\ZZ^\varepsilon$, and a functor $\ZZ^{\Sigma_-,\Sigma_+} :\cob_{n+2}^{\Sigma_-,\Sigma_+} \to \Hom_\CC(\ZZ^\varepsilon(\Sigma_-),\ZZ^\varepsilon(\Sigma_+))$ for every $\Sigma_-,\Sigma_+$. We will define these using Theorem \ref{Thm_JuhaszCorner}, with $F_{M,\S} = Z(\S)$. We then need to check that these, together with the coherence data of $\ZZ^\varepsilon$, assemble into a 2-functor.

\begin{definition}
    Let $\Sigma_-,\Sigma_+$ be $n$-manifolds. Denote by $$\ZZ^\varepsilon_{\Sigma_-,\Sigma_+}: \cob_{(n+1)+\varepsilon}^{\Sigma_-,\Sigma_+} \to \CC^{\Sigma_-,\Sigma_+}:=\Hom_\CC(\ZZ^\varepsilon(\Sigma_-),\ZZ^\varepsilon(\Sigma_+))$$ the functor induced by $\ZZ^\varepsilon$ on the Hom-spaces. For every framed sphere $\S\subseteq M$, denote $$F_{M,\S}^{\Sigma_-,\Sigma_+}:= Z(\S): \ZZ^\varepsilon_{\Sigma_-,\Sigma_+}(M) \to \ZZ^\varepsilon_{\Sigma_-,\Sigma_+}(M(\S))$$
    where $Z(\S)$ has been defined in Definition \ref{Def_Z(S)}.
\end{definition}
\begin{proposition}\label{Prop_ZonHoms}
    The morphisms $F_{M,\S}^{\Sigma_-,\Sigma_+}$ satisfy relations (2)--(5), hence there exists a unique functor 
    $$\ZZ^{\Sigma_-,\Sigma_+}: \cob_{n+2}^{\Sigma_-,\Sigma_+} \to \CC^{\Sigma_-,\Sigma_+}  \quad\quad (\text{resp. }\ZZ^{\Sigma_-,\Sigma_+}: \cob_{n+2}^{\Sigma_-,\Sigma_+, nc} \to \CC^{\Sigma_-,\Sigma_+})$$
    extending $\ZZ^\varepsilon_{\Sigma_-,\Sigma_+}$ and such that $\ZZ^{\Sigma_-,\Sigma_+}(W(\S))=F_{M,\S}^{\Sigma_-,\Sigma_+}$.
\end{proposition}
\begin{proof} The conclusion follows from Corollary \ref{Cor_Juhasz} once we checked relations (2)--(5). We will drop the superscripts $F^{\Sigma_-,\Sigma_+}$ in the proof and denote $F_d := \ZZ^\varepsilon_{\Sigma_-,\Sigma_+}(d)$ for a diffeomorphism $d$.
    \paragraph{Relation (2):} We need to check that $F_{M', \S'}\circ F_d = F_{d^\S}\circ F_{M,\S}$ for $d:M\to M'$ and $\S\subseteq M$, where $d^\S: M(\S) \to M'(d\circ \S)$ is the diffeomorphism induced by $d$ and $\S'=d\circ \S$. Remember from Definition \ref{Def_hdlAttachment} that $\S$ induces a decomposition 
    $$M \simeq (\id_{\Sigma_+}\sqcup\  S^{k-1}\times \D^{n+2-k}) \circ M_{\smallsetminus \S}\text{  and  } M(\S) := (\id_{\Sigma_+}\sqcup\  \D^{k}\times S^{n+1-k}) \circ M_{\smallsetminus \S}$$
    and $\S'$ induces a decomposition
    $$M' \simeq (\id_{\Sigma_+}\sqcup\  S^{k-1}\times \D^{n+2-k}) \circ M'_{\smallsetminus \S'}\text{  and  } M'(\S) := (\id_{\Sigma_+}\sqcup\  \D^{k}\times S^{n+1-k}) \circ M'_{\smallsetminus \S'}$$
    By construction, the diffeomorphism $d$ restricts to a diffeomorphism $d_{\smallsetminus \S}:M_{\smallsetminus \S}\to M'_{\smallsetminus \S'}$. In this decomposition, $d$ can be rewritten as $(\id\sqcup \id_{S^{k-1}\times \D^{n+2-k}})\circ d_{\smallsetminus \S}$. The induced diffeomorphism on surgered manifolds is $d^\S = (\id\sqcup \id_{\D^{k}\times S^{n+1-k}})\circ d_{\smallsetminus \S}:M(\S)\to M'(\S')$.
    Equation (2) now follows from
    \begin{equation*}
        \begin{tikzcd}
        \Sigma_- \ar[r, "M_{\smallsetminus \S}",""{name=M1,below}] \ar[d,equal] & \Sigma_+\sqcup S^{k-1}\hskip-2pt\times\hskip-2pt S^{n+1-k}\hskip-4pt\ar[d,equal] \ar[rr, " \id\sqcup S\times \D",""{name=H1,below}]& &\Sigma_+\ar[d,equal] \\
        \Sigma_- \ar[r, "M'_{\smallsetminus \S'}",""{name=M2,below}]\ar[d,equal]  & \Sigma_+\sqcup S^{k-1}\hskip-2pt\times\hskip-2pt S^{n+1-k}\hskip-4pt\ar[d,equal]  \ar[rr, " \id\sqcup S\times \D",""{name=H2,below}]& &\Sigma_+ \ar[d,equal] \\
        \Sigma_- \ar[r, "M'_{\smallsetminus \S'}",""{name=M3,below}] & \Sigma_+\sqcup S^{k-1}\hskip-2pt\times\hskip-2pt S^{n+1-k}\hskip-4pt \ar[rr, " \id\sqcup \D\times S",""{name=H3,below}]& &\Sigma_+   
        \ar[from=M1,to=M2,Rightarrow,"d_{\smallsetminus \S}" near start, shorten > = 15]
        \ar[from=M2,to=M3,equal,"\id" near start, shorten > = 15]
        \ar[from=H1,to=H2,equal,"\id" near start, shorten > = 10]
        \ar[from=H2,to=H3,Rightarrow,"Z_k" near start, shorten > = 10]
    \end{tikzcd} = 
        \begin{tikzcd}
        \Sigma_- \ar[r, "M_{\smallsetminus \S}",""{name=M1,below}] \ar[d,equal] & \Sigma_+\sqcup S^{k-1}\hskip-2pt\times\hskip-2pt S^{n+1-k}\hskip-4pt\ar[d,equal] \ar[rr, " \id\sqcup S\times \D",""{name=H1,below}]& &\Sigma_+\ar[d,equal] \\
        \Sigma_- \ar[r, "M_{\smallsetminus \S}",""{name=M2,below}]\ar[d,equal]  & \Sigma_+\sqcup S^{k-1}\hskip-2pt\times\hskip-2pt S^{n+1-k}\hskip-4pt\ar[d,equal]  \ar[rr, " \id\sqcup S\times \D",""{name=H2,below}]& &\Sigma_+ \ar[d,equal] \\
        \Sigma_- \ar[r, "M'_{\smallsetminus \S'}",""{name=M3,below}] & \Sigma_+\sqcup S^{k-1}\hskip-2pt\times\hskip-2pt S^{n+1-k}\hskip-4pt \ar[rr, " \id\sqcup \D\times S",""{name=H3,below}]& &\Sigma_+   
        \ar[from=M1,to=M2,equal,"\id" near start, shorten > = 15]
        \ar[from=M2,to=M3,Rightarrow,"d_{\smallsetminus \S}" near start, shorten > = 15]
        \ar[from=H1,to=H2,Rightarrow,"Z_k" near start, shorten > = 10]
        \ar[from=H2,to=H3,equal,"\id" near start, shorten > = 10]
    \end{tikzcd}
    \end{equation*}
    where we have omitted writing $\ZZ^\varepsilon(-)$ to help fit in the page, and $S\times\D$ and $\D\times S$ respectively stand for $S^{k-1}\times \D^{n+2-k}$ and $\D^{k}\times S^{n+1-k}$.

    \paragraph{Relation (3):} We need to check that $F_{M(\S),\S'}\circ F_{M,\S}= F_{M(\S'),\S}\circ F_{M,\S'}$ whenever $\S,\S' \subseteq M$ are disjoint, so $\S'$ can be pushed to $M(\S)$ and $\S$ to $M(\S')$. The embedding $$\S\sqcup\S': S^{k-1}\times \D^{n+2-k} \sqcup\ S^{l-1}\times \D^{n+2-l} \inj M$$ induces a decomposition
    $$M \simeq (\id_{\Sigma_+}\sqcup\  S^{k-1}\times \D^{n+2-k}\sqcup\ S^{l-1}\times \D^{n+2-l}) \circ M_{\smallsetminus \S\cup\S'}$$
    Relation (3) follows from 
    \begin{equation*}
        \begin{tikzcd}
        \Sigma_- \ar[r, "M_{\smallsetminus \S\cup\S'}",""{name=M1,below}] \ar[d,equal] & \hskip-4pt\Sigma_+\sqcup \partial \S \sqcup \partial \S'\ar[d,equal] \ar[rr, " \id\sqcup S\times \D\sqcup S\times \D",""{near end, name=H1,below}]& &\Sigma_+\ar[d,equal] \\
        \Sigma_- \ar[r, "M_{\smallsetminus \S\cup\S'}",""{name=M2,below}]\ar[d,equal]  & \hskip-4pt\Sigma_+\sqcup \partial \S \sqcup \partial \S'\ar[d,equal]  \ar[rr, " \id\sqcup \D\times S\sqcup S\times \D",""{near end, name=H2,below}]& &\Sigma_+ \ar[d,equal] \\
        \Sigma_- \ar[r, "M_{\smallsetminus \S\cup\S'}",""{name=M3,below}] & \hskip-4pt\Sigma_+\sqcup \partial \S \sqcup \partial \S' \ar[rr, " \id\sqcup \D\times S\sqcup \D\times S",""{near end, name=H3,below}]& &\Sigma_+   
        \ar[from=M1,to=M2,equal,"\id" near start, shorten > = 10]
        \ar[from=M2,to=M3,equal,"\id" near start, shorten > = 10]
        \ar[from=H1,to=H2,Rightarrow,"\id\sqcup Z_k\sqcup \id"' near start, shorten > = 10]
        \ar[from=H2,to=H3,Rightarrow,"\id\sqcup\id\sqcup Z_l"' near start, shorten > = 10]
    \end{tikzcd} = 
    \begin{tikzcd}
        \Sigma_- \ar[r, "M_{\smallsetminus \S\cup\S'}",""{name=M1,below}] \ar[d,equal] & \hskip-4pt\Sigma_+\sqcup \partial \S \sqcup \partial \S'\ar[d,equal] \ar[rr, " \id\sqcup S\times \D\sqcup S\times \D",""{near end, name=H1,below}]& &\Sigma_+\ar[d,equal] \\
        \Sigma_- \ar[r, "M_{\smallsetminus \S\cup\S'}",""{name=M2,below}]\ar[d,equal]  & \hskip-4pt\Sigma_+\sqcup \partial \S \sqcup \partial \S'\ar[d,equal]  \ar[rr, " \id\sqcup S\times \D\sqcup \D\times S",""{near end, name=H2,below}]& &\Sigma_+ \ar[d,equal] \\
        \Sigma_- \ar[r, "M_{\smallsetminus \S\cup\S'}",""{name=M3,below}] & \hskip-4pt\Sigma_+\sqcup \partial \S \sqcup \partial \S' \ar[rr, " \id\sqcup \D\times S\sqcup \D\times S",""{near end, name=H3,below}]& &\Sigma_+   
        \ar[from=M1,to=M2,equal,"\id" near start, shorten > = 10]
        \ar[from=M2,to=M3,equal,"\id" near start, shorten > = 10]
        \ar[from=H1,to=H2,Rightarrow,"\id\sqcup \id\sqcup Z_l"' near start, shorten > = 10]
        \ar[from=H2,to=H3,Rightarrow,"\id\sqcup Z_k\sqcup \id"' near start, shorten > = 10]
    \end{tikzcd}
    \end{equation*}
        where again we have omitted writing $\ZZ^\varepsilon(-)$, wrote $\id\sqcup Z_k\sqcup \id$ when we mean $\id\otimes Z_k\otimes \id$ with inserted compatibility of $\ZZ^\varepsilon$ with the monoidal structure, and $\partial \S$ and $\partial \S'$ respectively stand for $S^{k-1}\times S^{n+1-k}$ and $S^{l-1}\times S^{n+1-l}$.

    \paragraph{Relation (4):} We need to check that $F_{M(\S),\S'}\circ F_{M,\S} = F_\varphi$ whenever $a(\S')$ and $b(\S)$ intersect transversely exactly once in $M(\S)$, and $\varphi$ is the induced diffeomorphism $M\simeq M(\S)(\S')$. 

    As in \cite[Def. 2.17]{Juhasz}, up to shrinking $\S$ and $\S'$ which does not affect $F_{M(\S),\S'}$ and $F_{M,\S}$ by relations (1) and (2), we can assume that $$\im(\S) \cup (\im(\S')\cap M) \simeq  \D^k\times S^{n+1-k} \underset{S^{k-1}\times \D^{n+1-k}}{\cup} \D^k\times \D^{n+1-k}\ .$$ Therefore a small neighborhood of this subset in $M$ is diffeomorphic to $B$ from Definition \ref{Def_TheHdlCancel} and decomposes $M$ as 
    $$M\simeq (\id_{\Sigma_+} \sqcup B) \circ M_{\smallsetminus B}$$
Moreover, $\S$ factors through $B$ and $\S'$ through $B(\S)$. If the algebraic count $a(\S')\pitchfork b(\S) = 1$, then they agree with the framed spheres $\S_B$ and $\S_B'$ described in Definition \ref{Def_TheHdlCancel}. However, if $a(\S')\pitchfork b(\S) = -1$ then they agree with $\S_B$ and $\overline{\S_B'}$ (or $\overline{\S_B}$ and $\S_B'$). We reduce to the first case using relation (5) below\footnote{We may use reversal of $\S$ for $k\neq 0$ or of $\S'$ for $k+1 \neq n+2$. This argument would fail for $n=-1$, which we do not consider.}.

The diffeomorphism $\varphi:M\to M(\S')(\S)$ is given by the diffeomorphism $\varphi_B:B\to B(\S')(\S)$ extended by the identity on $M_{\smallsetminus B}$.

We have $$F_{M(\S),\S'}\circ F_{M,\S} = (F_{B(\S_B),\S_B'}\circ F_{B,\S_B}) \circ_h \id_{M_{\smallsetminus B}}$$
which, as the $Z_k$'s satisfy the handle cancellation, implies
$$F_{M(\S),\S'}\circ F_{M,\S} = F_{\varphi_B} \circ_h \id_{M_{\smallsetminus B}} = F_\varphi \ .$$

    \paragraph{Relation (5):} We need to check that $F_{M,\S} = F_{M,\overline{\S}}$ where $\S$ is a framed $(k-1)$-sphere for some $1\leq k\leq n+1$ and $\overline{\S} = \S\circ \iota$. This follows from the definition of $Z_k$ being $\iota$-invariant.
\end{proof}
\begin{definition}
    Let $\ZZ:\Cob_{n+1+1}\to \CC$ (resp. $\Cob_{n+1+1}^{nc}\to \CC$) be the assignment \begin{itemize}
        \item $\ZZ(\Sigma) := \ZZ^\varepsilon(\Sigma)$ for $\Sigma$ an object of $\Cob_{n+1+1}$,
        \item $\ZZ^{\Sigma_-,\Sigma_+}: \Hom_{\Cob_{n+1+1}}(\Sigma_-,\Sigma_+) \to \Hom_\CC(\ZZ(\Sigma_-),\ZZ(\Sigma_+))$ defined by Proposition \ref{Prop_ZonHoms}, and
        \item Coherence structure $\phi_{M,M'}: \ZZ(M')\circ \ZZ(M) \simeq \ZZ(M'\circ M)$ and $\phi_{\Sigma}: \id_{\ZZ(\Sigma)} \simeq \ZZ(\id_\Sigma)$, and coherences for the symmetric monoidal structure, given by the coherence structure of $\ZZ^\varepsilon$.
    \end{itemize}
\end{definition}
\begin{proposition}\label{Prop_sm2fun}
    The assignment $\ZZ:\Cob_{n+1+1}\to \CC$ (resp. $\Cob_{n+1+1}^{nc}\to \CC$) is a symmetric monoidal 2-functor.
\end{proposition}
\begin{proof}
    We need to check that the coherence data of $\ZZ^\varepsilon$ still satisfies the required coherence properties of a symmetric monoidal 2-functor with respect to the new assignments on 2-morphisms given by $\ZZ$. We follow the definitions from \cite[Def. A.5 and Def. 2.5] {SPPhD}. Most of the coherences of this data do not involve 2-morphisms and hold by assumption that $\ZZ^\varepsilon$ is a symmetric monoidal 2-functor. We are left to check the following.

For every pair of composeable 1-morphisms $\Sigma_1 \overset{M}{\to}\Sigma_2\overset{M'}{\to}\Sigma_3$, we are given an isomorphism 
$$\phi_{M,M'}: \ZZ(M')\circ_h \ZZ(M) \simeq \ZZ(M'\circ_h M)$$
which by assumption is natural with respect to diffeomorphisms of $M$ and $M'$ (i.e. in the 2-morphisms coming from $\Cob_{n+1+\varepsilon}$). We need to check that it is actually natural with respect to all 2-morphisms. It is enough to check it for 2-morphisms of the form $W(\S)$ as they generate along with diffeomorphisms.

Let $\S\subseteq M'$ be a framed $(k-1)$-sphere and $W(\S): M'\to M'(\S)$. Naturality of $\phi$ follows from the construction of $\ZZ(W(\S))$:
\begin{equation*}
\begin{tikzcd}
    \ZZ(\Sigma_1)\ar[d, equal] \ar[r, "\ZZ(M)"] 
        & \ZZ(\Sigma_2)\ar[r, "\ZZ(M')"] 
        & \ZZ(\Sigma_3)\ar[d, equal]
        \\
        \ZZ(\Sigma_1)\ar[d, equal] \ar[rr, "\ZZ(M'\circ M)"{name = MM, above}] 
        && \ZZ(\Sigma_3)\ar[d, equal]
        \\
        \ZZ(\Sigma_1) \ar[rr, "\ZZ(M'(\S)\circ M)"{name = MSM, above}] 
        && \ZZ(\Sigma_3)
        \arrow[from = 1-2, to = MM, Rightarrow, "\phi_{M,M'}"]
        \arrow[from = MM, to = MSM, Rightarrow, "\ZZ(W(\S)\circ_h \id_M)"{near end}, shorten < = 7]
\end{tikzcd}
=
    \begin{tikzcd}[column sep=45pt]
        \ZZ(\Sigma_1)\ar[d, equal] \ar[r, "\ZZ(M)"] 
        & \ZZ(\Sigma_2)\ar[r, "\ZZ(M')"] 
        & \ZZ(\Sigma_3)\ar[d, equal]
        \\
        \ZZ(\Sigma_1)\ar[d, equal] \ar[rr, "\ZZ(M'\circ M)"{name = MM, above}] 
        && \ZZ(\Sigma_3)\ar[d, equal]
        \\
        \ZZ(\Sigma_1)\ar[d, equal] \ar[r, "\ZZ(M'_{\smallsetminus \S}\circ M)"{name = MmS2, above}] 
        & \ZZ(\Sigma_2\sqcup S\hskip-3pt\times\hskip-3pt S)\ar[d, equal]\ar[r, "\ZZ(\id_{\Sigma'}\sqcup S\times \D)"{name = SD, above}] 
        & \ZZ(\Sigma_3)\ar[d, equal]
        \\
        \ZZ(\Sigma_1)\ar[d, equal] \ar[r, "\ZZ(M'_{\smallsetminus \S}\circ M)"{name = MmS1, above}] 
        & \ZZ(\Sigma_2\sqcup S\hskip-3pt\times\hskip-3pt S)\ar[r, "\ZZ(\id_{\Sigma'}\sqcup \D\times S)"{name = DS, above}] 
        & \ZZ(\Sigma_3)\ar[d, equal]
        \\
        \ZZ(\Sigma_1) \ar[rr, "\ZZ(M'(\S)\circ M)"{name = MSM, above}] 
        && \ZZ(\Sigma_3)
        \arrow[from = 1-2, to = MM, Rightarrow, "\phi"]
        \arrow[from = MM, to = 3-2, Rightarrow, "\phi^{-1}", shorten < = 7]
        \arrow[from = 4-2, to = MSM, Rightarrow, "\phi"]
        \arrow[from = SD, to = DS, Rightarrow, "\id\otimes Z_k"{near end}, shorten < = 7]
        \arrow[from = MmS1, to = MmS2, equal, shorten > = 7]
    \end{tikzcd}
\end{equation*}
which by compatibility of $\phi$ with associators and unitors is equal to
\begin{equation*}
    \begin{tikzcd}
        \ZZ(\Sigma_1)\ar[d, equal] \ar[r, "\ZZ(M)"{name = M1, above}] 
        & \ZZ(\Sigma_2)\ar[d, equal]\ar[rr, "\ZZ(M')"{name = M, above}] 
        && \ZZ(\Sigma_3)\ar[d, equal]
        \\
        \ZZ(\Sigma_1)\ar[d, equal]\ar[r, "\ZZ(M)"{name = M2, above}] 
        & \ZZ(\Sigma_2)\ar[d, equal] \ar[r, "\ZZ(M'_{\smallsetminus \S})"{name = MmS2, above}] 
        & \ZZ(\Sigma_2\sqcup S\hskip-3pt\times\hskip-3pt S)\hskip5pt\ar[d, equal]\ar[r, "\scriptstyle \ZZ(\id\sqcup S\times \D)"{name = SD, above}] 
        & \hskip5pt\ZZ(\Sigma_3)\ar[d, equal]
        \\
        \ZZ(\Sigma_1)\ar[d, equal]\ar[r, "\ZZ(M)"{name = M3, above}] 
        & \ZZ(\Sigma_2)\ar[d, equal] \ar[r, "\ZZ(M'_{\smallsetminus \S})"{name = MmS1, above}] 
        & \ZZ(\Sigma_2\sqcup S\hskip-3pt\times\hskip-3pt S)\hskip5pt\ar[r, "\scriptstyle \ZZ(\id\sqcup \D\times S)"{name = DS, above}] 
        & \hskip5pt\ZZ(\Sigma_3)\ar[d, equal]
        \\
        \ZZ(\Sigma_1)\ar[d, equal]\ar[r, "\ZZ(M)"{name = M4, above}] 
        & \ZZ(\Sigma_2) \ar[rr, "\ZZ(M'(\S))"{name = MS, above}] 
        && \ZZ(\Sigma_3)\ar[d, equal]
        \\
        \ZZ(\Sigma_1) \ar[rrr, "\ZZ(M'(\S)\circ M)"{name = MSM, above, near start}] 
        &&& \ZZ(\Sigma_3)
        \arrow[from = M, to = 2-3, Rightarrow, "\phi^{-1}", shorten < = 7]
        \arrow[from = 3-3, to = MS, Rightarrow, "\phi"]
        \arrow[from = 4-2, to = MSM, Rightarrow, "\phi"]
        \arrow[from = SD, to = DS, Rightarrow, "\id\otimes Z_k"{near end}, shorten < = 7]
        \arrow[from = MmS1, to = MmS2, equal, shorten > = 7]
        \arrow[from = M1, to = M2, equal, shorten < = 7]
        \arrow[from = M2, to = M3, equal, shorten < = 7]
        \arrow[from = M3, to = M4, equal, shorten < = 7]
    \end{tikzcd}
=
\begin{tikzcd}
    \ZZ(\Sigma_1)\ar[d, equal] \ar[r, "\ZZ(M)"{name = M1, above}] 
        & \ZZ(\Sigma_2)\ar[d, equal]\ar[r, "\ZZ(M')", ""{name = M, above, near start}] 
        & \ZZ(\Sigma_3)\ar[d, equal]
        \\
    \ZZ(\Sigma_1)\ar[d, equal] \ar[r, "\ZZ(M)"{name = M2, above}] 
        & \ZZ(\Sigma_2)\ar[r, "\ZZ(M'(\S))", ""{name = MS, above, near start}] 
        & \ZZ(\Sigma_3)\ar[d, equal]
        \\
        \ZZ(\Sigma_1) \ar[rr, "\ZZ(M'(\S)\circ M)"{name = MSM, above}] 
        && \ZZ(\Sigma_3)
        \arrow[from = 2-2, to = MSM, Rightarrow, "\phi_{M,M'(\S)}"]
        \arrow[from = M1, to = M2, equal, shorten < = 7]
        \arrow[from = M, to = MS, Rightarrow, "\ZZ(W(\S))", shorten  = 7]
\end{tikzcd}
\end{equation*}
Naturality with 2-morphisms in $M$ is similar. 
This proves that $\ZZ$ is indeed a 2-functor. We now have to prove that it is symmetric monoidal. 

For every pair of objects $\Sigma, \Sigma'$, we are given a 1-morphism $$\chi_{\Sigma, \Sigma'}: \ZZ(\Sigma)\otimes \ZZ(\Sigma') \to \ZZ(\Sigma\sqcup\Sigma')$$
and for every pair of 1-morphisms $\Sigma_- \overset{M}{\to}\Sigma_+,\ \Sigma_-'\overset{M'}{\to}\Sigma_+'$, we are given an isomorphism 
$$\chi_{M,M'}: \chi_{\Sigma_+,\Sigma_+'} \circ (\ZZ(M)\otimes \ZZ(M')) \Rightarrow \ZZ(M\sqcup M') \circ \chi_{\Sigma_-,\Sigma_-'}$$
which by assumption is natural with respect to diffeomorphisms. We need to check that it is actually natural with respect to all 2-morphisms and again it is enough to check it for 2-morphisms of the form $W(\S)$. The argument is similar to the one above, except that we have disjoint union of $M$ and $M'$ instead of side composition of $M$ and $M'$.
\end{proof}

\begin{proof}[Proof of Theorem \ref{Thm_ETQFTfromHandle}]
The existence of $\ZZ$ is provided by Proposition \ref{Prop_sm2fun}. It extends $\ZZ^\varepsilon$ and satisfies $\ZZ(H_k)=Z_k$ by construction. Unicity is given by the existence of parametrized Cerf decompositions for 2-morphisms. 

Reciprocally, we have seen that the 2-morphisms $(H_k)_k$ satisfy (a) and (b) in $\Cob_{n+1+1}$, hence so do $(\ZZ(H_k))_k$.
\end{proof}

\section{Classification of extensions}\label{Sec_Classif}
We will see that, given the $k$-handle $Z_k$, there is at most one 2-morphism $Z_{k+1}$ which will satisfy the handle cancellation and $\iota$-invariance. In particular, the extension $\ZZ$ of a categorified TQFT $\ZZ^\varepsilon$ is determined by it value on the 0-handle $Z_0$. This has been noticed and exploited in \cite{ReutterSlides, WalkerSlides}.

The idea of the proof is very similar to the argument in \cite[Prop. 3.4.19]{LurieCob} where it is shown that the $k$-handle is a unit for some adjunction, and the $(k+1)$-handle is the counit. In a bicategory, given the unit of an adjunction, there is only one counit which satisfies the snake relations, hence the $k$-handle determines the $(k+1)$-handle. This argument doesn't exactly work in our setting as the morphism for which these are unit and counit has corners and is not allowed as a 1-morphism in our cobordism bicategory, but the idea remains valid.
\begin{theorem}\label{Thm_Classif}
    Let $\ZZ,\ZZ': \Cob_{n+1+1}\to \CC$ (resp. $\Cob_{n+1+1}^{nc}\to \CC$) be (resp. non-compact) once-extended TQFTs which agree on $\Cob_{n+1+\varepsilon}$ and such that $\ZZ(H_0)=\ZZ'(H_0)$. Then $\ZZ=\ZZ'$.
\end{theorem}
\begin{proof}
    By the uniqueness in Theorem \ref{Thm_ETQFTfromHandle}, it suffices to show that $\ZZ(H_k)=\ZZ'(H_k)$ for every $k\leq n+2$ (resp. $k\leq n+1$). By induction, suppose that $Z_k := \ZZ(H_k)=\ZZ'(H_k)$. Denote $Z_{k+1} := \ZZ(H_k)$ and $Z_{k+1}':=\ZZ'(H_k)$.

We will construct a cobordism with one $k$-handle and two $(k+1)$-handles, for which we may want to use either $Z_{k+1}$ of $Z_{k+1}'$, so that we can cancel either one of them with the $k$-handle, and after cancellation the remaining $(k+1)$-handle is isomorphic to the standard handle. 

Consider $M := S^k\times \D^{n+1-k}$ and the $(k-1)$-sphere $\S: S^{k-1} \overset{-\times\{0\}}{\inj} S^k\times \D^{n+1-k}$, which is canonically framed by crossing a tubular neighborhood $\nu(S^{k-1})$ of $S^{k-1}$ in $S^k$ with $\D^{n+1-k}$ and smoothing the corners of the result. Note that $\nu(S^{k-1})\times\{0\} \simeq S^{k-1}\times[-1,1]$ intersects the boundary $S^{k-1}\times S^{n+1-k}$ of the framed sphere $\S$ along two copies of $S^{k-1}$, $S^{k-1}\times \{x_\pm\}$, for some antipodal $x_\pm \in S^{n+1-k}$, which we may call north and south poles.

The surgered manifold is $$M(\S) = (S^k\times \D^{n+1-k} \smallsetminus \im(\S))\underset{S^{k-1}\times S^{n+1-k}}{\cup} \D^k\times S^{n+1-k}$$ 
Denote $S^k_+$ and $S^k_-$ the two connected components of $S^k \smallsetminus \nu(S^{k-1})$. Then 
$$\S_\pm = S^k_\pm\times \{0\} \underset{S^{k-1}}{\cup}\D^k\times \{x_\pm\} \subseteq M(\S)$$
are disjoint $k$-spheres, which each intersect the belt sphere of $\S$ exactly once, transversely.
We orient them by taking the orientation induced by the $S^k_\pm$ part and they are canonically framed as above. 

As they are disjoint, we can think of $\S_\pm$ as a framed sphere in $M(\S)(\S_\mp)$. As $a(\S_\mp)$ intersect $b(\S)$ transversely once, we have a diffeomorphism $\varphi_\pm: M(\S)(\S_\mp) \simeq M$ supported in a neighborhood of $b(\S)\cup b(\S_\mp)$. The sphere $\S_\pm$ transported under this diffeomorphism is isotopic to the identity, as can be seen by sliding $\S_\pm$ along $\S_\mp$.

Let us denote $Z(\S_\pm): \ZZ(M(\S))\to \ZZ(M(\S)(\S_\pm))$, resp. $Z'(\S_\pm): \ZZ(M(\S))\to \ZZ(M(\S)(\S_\pm))$, the 2-morphism defined in Definition \ref{Def_Z(S)} using $Z_{k+1}$, resp $Z_{k+1}'$. They are respectively equal to $\ZZ(W(\S_\pm))$ and $\ZZ'(W(\S_\pm))$. Note that the 2-morphism $Z(\S)$ is equal to both $\ZZ(W(\S))$ and $\ZZ'(W(\S))$ by assumption, and similarly $\ZZ(W_\varphi)=\ZZ'(W_\varphi)$ below.
We have
\begin{multline*}
    Z(\S_+)\circ \ZZ(W_{\varphi_-}) 
    = Z(\S_+) \circ Z'(\S_-)\circ Z(\S) \overset{(2)}{=} Z'(\S_-) \circ Z(\S_+)\circ Z(\S)  = Z'(\S_-)\circ \ZZ(W_{\varphi_+})
\end{multline*}
where the first equality follows from the fact that $Z_k, Z_{k+1}'$ cancel, and the third from the fact that $Z_k, Z_{k+1}$ cancel. The second one follows from relation (2), i.e. commutativity of disjoint handles. 

As $Z_{k+1}'$ can be obtained from $Z'(\S_-):M(\S)(\S_+)\to M(\S)(\S_+)(\S_-)$ by composing with isomorphisms induced by diffeomorphisms, we obtain a formula to express $Z_{k+1}'$ in terms of $Z_{k+1}$ and diffeomorphisms. This formula also holds to express $Z_{k+1}$ in terms of $Z_{k+1}$ and diffeomorphisms by using $Z_{k+1}$ instead of $Z'_{k+1}$ in the equation above. Hence, $Z_{k+1} = Z'_{k+1}$.
\end{proof}
\begin{remark}
Suppose that we are given a categorified TQFT $\ZZ^\varepsilon$ and 2-morphisms $(Z_k)_k$ which satisfy 
\begin{itemize}
    \item[(a)] Each pair $Z_k,\ Z_{k+1}$ satisfies the handle cancellation with $a(\S')\pitchfork b(\S)=1$ as in Definition \ref{Def_TheHdlCancel}, and 
    \item[(a')] Each pair $Z_k,\ Z_{k+1}$ satisfies the handle cancellation with $a(\S')\pitchfork b(\S)=-1$.
\end{itemize} 
but not necessarily (b)  $\iota$-invariance. 

Consider the 2-morphism $(\overline{Z_k})_k$, where $\overline{Z_k}$ is the LHS of Definition \ref{Def_iotaInvariant} for $k \neq 0, n+2$, i.e. $\overline{Z_k} = Z(\iota: S^{k-1}\times \D^{n+2-k} \to S^{k-1}\times \D^{n+2-k})$ is the attachment associated to the reversed sphere, and $\overline{Z_k}=Z_k$ for $k = 0, n+2$. By assumption $Z_k, \overline{Z_{k+1}}$ cancel, for any sign of the algebraic intersection.

The construction above shows inductively that $Z_k = \overline{Z_k}$, i.e. we have $\iota$-invariance. 

We have already seen that $\iota$-invariance implies cancellation with algebraic intersection $-1$, hence (a) and (a') is equivalent to (a) and (b).
\end{remark}

\small
\bibliography{mybib.bib}
\bibliographystyle{alpha}
\end{document}